\documentclass[11pt,reqno]{amsart}

\usepackage[usenames, dvipsnames]{color}
\usepackage{amsmath, amsthm, amssymb}
\usepackage{amsfonts}
\usepackage[ansinew]{inputenc}
\usepackage{graphicx}
\usepackage[english]{babel}
\usepackage{enumerate}
\usepackage{hyperref}
\usepackage{stackrel}
\theoremstyle{plain}
\newtheorem{thm}{Theorem}[section]
\newtheorem{cor}[thm]{Corollary}
\newtheorem{lem}[thm]{Lemma}

\newtheorem{example}[thm]{Example}

\theoremstyle{definition}
\newtheorem{defi}[thm]{Definition}

\theoremstyle{remark}
\newtheorem{rem}[thm]{Remark}

\numberwithin{equation}{section}

\def \dist {\mathrm{dist}}

\newcommand{\R}{\mathbb{R}}

\newcommand{\defeq}{\mathrel{\mathop:}=}
\newcommand{\tr}{\mathrm{Tr}}

\newcommand{\average}{{\mathchoice {\kern1ex\vcenter{\hrule height.4pt
width 6pt depth0pt} \kern-9.7pt} {\kern1ex\vcenter{\hrule
height.4pt width 4.3pt depth0pt} \kern-7pt} {} {} }}

\def\R{\mathbb{R}}

\begin{document}

\title[Sharp regularity for degenerate obstacle type problems]{Sharp regularity for degenerate obstacle type problems: a geometric approach}

\author{Jo\~ao Vitor Da Silva}

\address{Departamento de Matem\'atica - Instituto de Ci\^{e}ncias Exatas - Universidade de Bras\'{i}lia.
\hfill\break \indent Campus Universit\'{a}rio Darcy Ribeiro, 70910-900, Bras\'{i}lia - Distrito Federal - Brazil.}
\email[J.V. da Silva]{J.V.Silva@mat.unb.br}

\address{Instituto de Investigaciones Matem\'{a}ticas Luis A. Santal\'{o} (IMAS) - CONICET,
\hfill\break \indent Ciudad Universitaria, Pabell\'{o}n I (1428) Av. Cantilo s/n - Buenos Aires, Argentina}
\email[J.V. da Silva]{jdasilva@dm.uba.ar}

\author{Hern\'an Vivas}

%\address{Instituto de Investigaciones Matem\'{a}ticas Luis A. Santal\'{o} (IMAS) - CONICET,
%\hfill\break \indent Ciudad Universitaria, Pabell\'{o}n I (1428) Av. Cantilo s/n - Buenos Aires, Argentina}

\address{Instituto de Investigaciones Matem\'{a}ticas Luis A. Santal\'{o} (IMAS) - CONICET,
\hfill\break \indent Ciudad Universitaria, Pabell\'{o}n I (1428) Av. Cantilo s/n - Buenos Aires, Argentina}

\address{Centro Marplatense de Investigaciones matem\'aticas/Conicet,
\hfill\break \indent Dean Funes 3350, 7600 Mar del Plata, Argentina.}
\email[Hern\'an Vivas]{havivas@mdp.edu.ar}


\begin{abstract}
We prove sharp regularity estimates for solutions of obstacle type problems driven by a class of degenerate fully nonlinear operators. More specifically, we consider viscosity solutions of
\begin{equation*}
\left\{
\begin{array}{rcll}
  |D u|^\gamma F(x, D^2u)& = & f(x)\chi_{\{u>\phi\}} & \textrm{ in } B_1 \\
  	u(x) & \geq & \phi(x) & \textrm{ in } B_1 \\
  u(x) & = & g(x) & \textrm{ on } \partial B_1,
\end{array}
\right.
\end{equation*}
with $\gamma>0$, $\phi \in C^{1, \alpha}(B_1)$ for some $\alpha\in(0,1]$, a continuous boundary datum $g$ and $f\in L^\infty(B_1)\cap C^0(B_1)$ and prove that they are $C^{1,\beta}(B_{1/2})$ (and in particular at free boundary points) where $\beta=\min\left\{\alpha, \frac{1}{\gamma+1}\right\}$. Moreover, we achieve such a feature by using a recently developed geometric approach which is a novelty for these kind of free boundary problems. Further, under a natural non-degeneracy assumption on the obstacle, we prove that the free boundary $\partial\{u>\phi\}$ has Hausdorff dimension less than $n$ (and in particular zero Lebesgue measure). Our results are new even for degenerate %quasilinear (En general Operadores Quasilinear son otra classe tales como el p-Laplacian or el g-Laplacian operators)
problems such as
\[
   |Du|^\gamma \Delta u=\chi_{\{u>\phi\}} \quad \text{with}\quad \gamma>0.
\]
\end{abstract}

\maketitle
%\tableofcontents

\section{Introduction}

The aim of this work is twofold: on one hand, to get sharp regularity estimates for solutions to an obstacle type problem involving a degenerate fully nonlinear operator. On the other hand, in order to achieve such purposes, we appeal to improvement of flatness technique and geometric estimates that have proved to be very useful in dealing with regularity issues for elliptic/parabolic PDEs in the last decade (see, \cite{AdaSRT19}, \cite{ART15}, \cite{ATU18}, \cite{APR17}, \cite{BeDe14}, \cite{DD19}, \cite{daSLR}, \cite{IS13} and \cite{ST15} for some enlightening examples), but that, to the best of the authors' scientific knowledge, have not been applied to free boundary problems of obstacle type as the one dealt with in this paper.

To be more precise, in this manuscript we study geometric regularity estimates for obstacle type problems governed by second order nonlinear elliptic operators of degenerate type as follows:
\begin{equation}\label{Eq1}
\left\{
\begin{array}{rcll}
  |D u|^\gamma F(x, D^2u)& = & f(x)\chi_{\{u>\phi\}} & \textrm{ in } B_1 \\
  u & \geq & \phi & \textrm{ in } B_1 \\
  u & = & g & \textrm{ on } \partial B_1,
\end{array}
\right.
\end{equation}
where $\phi$ is a (suitably regular) obstacle, $g$ is a continuous boundary datum, $f$ is a bounded and continuous function, $\gamma>0$, $\chi_E$ stands for the characteristic function of the set $E$ and $F:B_1\times \text{Sym}(n) \longrightarrow \R$ is a second order fully nonlinear operator which is uniformly elliptic and satisfies minimal continuity assumptions on coefficients to be presented soon (recall that $\text{Sym}(n)$ is the space of symmetric matrices in $\R^n$). Hereafter \eqref{Eq1} will be referred to as the $(F, \gamma,\phi, f)$-obstacle problem.

\subsection{Obstacle problems for non-divergence form operators}

In Mathematical Physics the classical obstacle problem refers to the equilibrium position of an elastic membrane (whose boundary is held fixed) lying above a given barrier (an obstacle) and subject to the action of external forces \textit{e.g}. friction, tension, air resistance and/or gravity. In its most simplified model, the height of such a membrane fulfils the following problem (in a suitable weak sense):
$$
\left\{
\begin{array}{rcll}
  \Delta u& = & \chi_{\{u>0\}} & \textrm{ in } \Omega \\
  u & \geq & 0 & \textrm{ in } \Omega \\
  u & = & g & \textrm{ on } \partial \Omega,
\end{array}
\right.
$$
where $\Omega \subset \R^n$ is a given domain and $g$ is a regular boundary datum. Notice that in such a setting the right-hand side of above (first) equation processes ``a jump'' across the \textit{a priori} unknown interface $\partial\{u >0\}$, the so-named \textit{the free boundary}.

We should remember that obstacle type problems have attracted an increasing enthusiasm of the multidisciplinary scientific community for the last five decades or so. One of the main reasons, besides their intrinsic mathematical appeal that combines tools from regularity theory for PDEs, Calculus of Variations, Geometric Measure Theory, Nonlinear Potential Theory and Harmonic Analysis, is that they are ubiquitous in Sciences, Mechanics, Engineering and Industry. In fact, problems as varied as flow through porous dam, cellular membranes' permeability, optimal stopping problems in Mathematical Finance, superconductivity of bodies in mean-field models in Physics are just some examples of phenomena that appear to be well described by these type of problems. We refer the reader to \cite{PSU} and the references therein for an instrumental survey and progresses of such investigations concerning obstacle problems with divergence structure.

Concerning elliptic nonlinear problems in non-divergence form, when the solutions are assumed to be non-negative, the works \cite{L} and \cite{LS01} address a complete study on obstacle type problems in the fully nonlinear (uniformly elliptic) scenario with homogeneous obstacles and/or source terms and their corresponding regularity theories of solutions and free boundaries. Recently, Blank and Teka in \cite{BlankTeka} deal with strong solutions $w\geq 0$ of an obstacle problem of the form
$$
    \mathfrak{L}\, w(x) = a^{ij}(x)D_{ij} w(x) = \chi_{\{w>0\}} \quad \text{in} \quad B_1.
$$
In such a context, by assuming that $a^{ij} \in \text{VMO}(\Omega)$ (and uniform ellipticity), the authors prove existence of nontrivial solutions, non-degeneracy and optimal regularity of solutions. Finally, regarding inhomogeneous elliptic obstacle problems in \textit{non-divergence form} (with possibly discontinuous coefficients and irregular obstacles and nonlinearities), we must quote \cite{BLOP}, where the authors study the following problem:
$$
  \left\{
  \begin{array}{rcll}
  F(x, D^2 u) &\leq & f & \textrm{ in } \Omega \\
  (F(x, D^2 u)-f)(u-\psi) & = & 0 & \textrm{ in } \Omega \\
  u&\geq &\psi&\textrm{ in } \Omega\\
  u & = & 0 & \textrm{ on } \partial \Omega,
\end{array}\right.
$$
where the fully nonlinear operator $F: \Omega \times \text{Sym}(n) \to \R$ is supposed to be uniformly elliptic (with $X \mapsto F(x, X)$ a convex mapping with $F(x, 0) = 0$ for all $x \in \Omega$). In such a scenario, they investigate existence/uniqueness and regularity issues of solutions. Of particular interest, they establish H\"{o}lder continuity of the gradient of the solutions (see, \cite[Theorem 5.4]{BLOP}).

From a strictly mathematical perspective, operators such as the one in \eqref{Eq1} are essentially the simplest (non translation invariant) example of a more general class of degenerate operators that have attracted much attention in the PDE community over the last decade and a half. Particularly, we recommend the reading of Birindelli-Demengel's fundamental works \cite{BerDem} and \cite{BeDe14} for a number of interesting examples of such operators.

It is worth highlighting that some of major difficulties in dealing with such a class of operators are: its non-divergence structure, in consequence of which we are not allowed to make use of (nowadays) classical estimates from nonlinear potential theory/harmonic analysis (see \cite{ALS15} or \cite{PSU}), and its degeneracy character, which implies that diffusion properties (e.g., uniform ellipticity of operator) collapse along an \textit{a priori} unknown set of singular points of solutions, namely  the set $\mathcal{S}_0 \defeq \{x\in B_1 : |Du(x)| = 0\}.$ % Dejemos la notación \mathcal{S}_0 pues la utilizamos en nuestro otro paper y también aparecere en el paper \cite{ART15}. Entonces és algo ya standard en la teoría.

The first feature (the non-divergence structure characteristic of fully nonlinear equations) is already present in the non-degenerate case. Indeed, if we consider the following equation
\[
  F(D^2u)=0\quad\textrm{ in } \quad B_1
\]
for a uniformly elliptic operator $F$, one standard way to achieve regularity is via a formal linearization process, where both $u$ and its first derivatives satisfy linear elliptic equations in non-divergence form with bounded measurable coefficients. In this context, Krylov-Safonov's Harnack inequality yields that solutions are $C_{\text{loc}}^{1,\alpha_F}$ for some universal, but unknown $\alpha_F \in (0, 1]$. The viscosity solutions' language allows us to obtain similar conclusions without appealing to such a linearization procedure, see for instance \cite[Ch.5 \S 3]{CC95}. Nevertheless, to get higher regularity such as the desired $C^2$, that would make solutions classical, a further restriction has to be imposed on the operator. Roughly speaking, if we wish to repeat the formal process described above to get information on second derivatives of $u$, we need to enforce a sign constraint on the second derivatives of $F$ with respect to $X$. As a matter of fact, under a concavity (or convexity) assumption the seminal papers of Evans \cite{Ev82} and Krylov \cite{Kry83} provide local $C^{2,\alpha}$ estimates (see, \cite[Ch.6 \S 1]{CC95} for a slightly refined proof). On the other hand, if no assumption other than uniform ellipticity is imposed, solutions are not better than $C^{1,\alpha_F}$, as was addressed by Nadirashvili-Vl\u{a}du\c{t} in \cite{NV1} and some subsequent works (even thought $F$ is a smooth operator).

The degenerate character of our operators makes the situation even more delicate. Given $\gamma>0$, $F$ a uniformly elliptic operator and $f\in L^\infty(B_1)\cap C^0(B_1)$ we can study the degenerate problem
\begin{equation}\label{eq.deg}
|D u|^\gamma F(D^2 u)=f(x)\quad\textrm{ in } \quad B_1.
\end{equation}
Imbert and Silvestre in \cite[Theorem 1]{IS13} showed that solutions to \eqref{eq.deg} are $C^{1, \alpha}$ for some (small) $\alpha \in (0, 1)$. In the aftermath of these studies, in \cite[Theorem 3.1]{ART15}, Ara\'{u}jo, Ricarte and Teixeira showed that in fact, given $\beta\in (0,\alpha_F)\cap\left(0,\frac{1}{1+\gamma}\right]$ then $u\in C^{1,\beta}\left(B_{1/2}\right)$ and that this result is optimal and holds for operators that are not necessarily translation invariant. In particular, if $F$ is convex or concave, solutions belong to $C^{1,\frac{1}{1+\gamma}}$ in the interior (see, \cite[Corollary 3.2]{ART15}).

Considering all of the above, the problem we wish to address here can be summarized as follows: given $\gamma>0$, $F$ a convex or concave uniformly elliptic operator satisfying assumptions \eqref{UnifEllip} and \eqref{eqxdep} (related to the $x$ dependence) in Section \ref{ssec.defandmain}\footnote{For sake of  simplicity, we will make the normalizing assumption that $F(x, 0) = 0$ for all $x \in B_1$, which does not entail any loss of generality. (cf. \cite{BLOP} and \cite{DSV}).}, $g\in C(\partial B_1)$ and $\phi\in C^{1,\alpha}(\overline B_1)$ with $\alpha\in\left(0,1\right]$ (and $\phi<g$ on $\partial B_1$) and $u$ a viscosity solution of \eqref{Eq1}. %In this point, the following issue arises:
How smooth is such a solution? Our (sharp) result says that $u\in C^{1,\beta}\left(B_{1/2}\right)$ for
\begin{equation}\label{eq.beta}
\beta\defeq \min\left\{\alpha,\frac{1}{1+\gamma}\right\}.
\end{equation}

\subsection{Definitions and main results}\label{ssec.defandmain}

In this section we present some definitions that are needed and present the main results of the paper. First, we recall that, for second order operators, \textit{uniform ellipticity} means that for any pair of matrices $X,Y\in Sym(n)$
\begin{equation}\label{UnifEllip}
    \mathcal{M}_{\lambda, \Lambda}^-(X-Y)\leq F(x, X)-F(x, Y)\leq\mathcal{M}_{\lambda, \Lambda}^+(X-Y)
\end{equation}
where $\mathcal{M}_{\lambda, \Lambda}^-$ and $\mathcal{M}_{\lambda, \Lambda}^+$ are the \textit{Pucci extremal operators} given by
\[
   \mathcal{M}_{\lambda, \Lambda}^-(X)=\lambda\sum_{e_i>0}e_i+\Lambda\sum_{e_i<0}e_i\quad\textrm{ and }\quad \mathcal{M}_{\lambda, \Lambda}^+(X)=\Lambda\sum_{e_i>0}e_i+\lambda\sum_{e_i<0}e_i
\]
for some \emph{ellipticity constants} $0<\lambda\leq \Lambda< \infty$ (here $\{e_i\}_i$ are the eigenvalues of $X$).

In the sequel, it is necessary to impose some restriction in the behavior of the ``coefficients'' our equations, that is the $x$ dependence of the operators we are going to consider. Following \cite{ART15} and \cite{BeDe14} we are going to prescribe that there is a uniform modulus of continuity of the ``coefficients'':
\begin{equation}\label{eqxdep}
\Theta_F(x, y) \defeq \sup_{\substack{X \in  Sym(n) \\ X \neq 0}}\frac{|F(x,X)-F(y,X)|}{\|X\|}\leq C_F\omega(|x-y|)
\end{equation}
where $C_F>0$ is some constant and $\omega$ is some modulus of continuity, i.e. a positive increasing function in $\R^+$ satisfying
\[
\lim_{r\rightarrow0^+} \omega(r)=0.
\]

Let us now review the definition of viscosity solution for our operators. For $G: B_1 \times \R^n \times Sym(n) \longrightarrow \R$ and $f: B_1 \to \R$ continuous functions we have the following definition (for $G$ a non-decreasing function with respect to the partial order of symmetric matrices):

\begin{defi}[{\bf Viscosity solutions, \cite{ART15} and \cite{BeDe14}}] $u \in C^{0}(B_1)$ is a viscosity super-solution (resp. sub-solution) to
$$
 G(x, Du, D^2 u) = f(x) \quad \mbox{in} \quad B_1
$$
if for every $x_0 \in B_1$ we have the following:  $\forall$ $\varphi \in C^2(B_1)$ such that $u-\phi$ has a local minimum at $x_0$ there holds
$$
      G(x_0, D\varphi(x_0), D^2 \varphi(x_0)) \leq  f(x_0) \quad (resp. \,\, \geq f(x_0)).
$$
\end{defi}

As measure of the smoothness of solutions along free boundary points, we are going to use the following norm:
\begin{defi}[{\bf $C^{1, \alpha}$ norm}] For $\alpha \in (0,1]$, $C^{1, \alpha}(B_1)$ denotes the space of $u$ whose spacial gradient $Du(x)$ there exists in the classical sense for every $x\in B_1$ and such that
$$
\|w\|_{C^{1, \alpha}(B_1)} \defeq \|w\|_{L^{\infty}(B_1)} + \|Dw\|_{L^{\infty}(B_1)}
		 	+ \displaystyle \sup_{\substack{x,y\in B_1 \\x \not= y}} \frac{|w(x)-[w(y)-D w(y)\cdot (x-y)]|}{|x -y|^{1+\alpha}}
$$
is finite. It is easy to verify that $w \in C^{1, \alpha}(B_1)$ implies every component of $Dw$ is $C^{0, \alpha}(B_1)$.
\end{defi}

Since we are interested on the regularity at free boundary points, we will often assume without loss of generality that $0\in\partial\{u>\phi\}$ and perform our estimates there.

Finally, a rather weak non-degeneracy property is going to be imposed on the obstacle throughout the paper:
\begin{equation}\label{eq.assumption}
0\leq |D\phi|^{\gamma}F(x, D^2 \phi) \leq \|f\|_{L^{\infty}(B_1)} \quad \text{in} \quad B_1
\end{equation}
in the viscosity sense. It is worth mentioning that the right hand side of this inequality is used just in the Improvement of Flatness Lemma \ref{lem.flat} and could be dispensed for the existence theory, whereas the left hand side is necessary for using the Comparison Principle in the existence theory laid out in the Appendix and is not needed for regularity purposes.

Now we are in a position to state our main results. The first result establishes an optimal growth estimate at free boundary points. In effect, it states that if the obstacle is $C^{1,\alpha}(B_1)$ smooth and the source term is bounded, then any bounded viscosity solution to the $(F, \gamma,\phi, f)$-obstacle problem in $B_1$ is $C^{1, \min\left\{\alpha, \frac{1}{\gamma+1}\right\}}$ at free boundary points.

\begin{thm}[{\bf Regularity along free boundary points}]\label{main} Suppose that the assumptions \eqref{UnifEllip} and \eqref{eqxdep} are in force for a convex or concave operator $F$ and let $\alpha\in(0,1]$. Let $u$ be a bounded viscosity solution to the $(F, \gamma,\phi, f)$-obstacle problem with obstacle $\phi \in C^{1, \alpha}(B_1)$ satisfying the second inequality of \eqref{eq.assumption} and $f \in L^\infty(B_1) \cap C^0(B_1)$.

Then, $u$ is $C^{1, \beta}(B_{1/2})$, and in particular at free boundary points, for $\beta$ satisfying \eqref{eq.beta}. More precisely, for any point $x_0 \in \partial \{u>\phi\} \cap B_{1/2}$ there holds
{\scriptsize{
\begin{equation}\label{EqSRE}
       \displaystyle \sup_{B_r(x_0)} \frac{|u(x)-(u(x_0)+ D u(x_0)\cdot (x-x_0))|}{r^{1+\beta}}\leq C.\left[\|u\|_{L^{\infty}(B_1)} + \left(\|\phi\|_{C^{1, \alpha}(B_1)}^{\gamma+1}+\|f\|_{L^{\infty}(B_1)}\right)^{\frac{1}{\gamma+1}}\right]
\end{equation}
}}
for $0<r < \frac{1}{2}$ where $C>0$ is a universal constant.

In particular,
\begin{equation}\label{detach}
\sup_{B_r(x_0)} \frac{u(x)-\phi(x)}{r^{1+\beta}}\leq C^{\ast}.\left[\|u\|_{L^{\infty}(B_1)} + \left(\|\phi\|_{C^{1, \alpha}(B_1)}^{\gamma+1}+\|f\|_{L^{\infty}(B_1)}\right)^{\frac{1}{\gamma+1}}\right],
\end{equation}
where $C^{\ast}>0$ is a universal constant, i.e. $u$ detaches from the obstacle at the speed dictated by the obstacle's modulus of continuity.
\end{thm}

It is noteworthy that our contributions extend/generalize regarding non-variational scenario, former  results (sharp regularity estimates) from \cite{BlankTeka}, and to some extent, of those from \cite{L} and \cite{LS01} (see also \cite{ALS15} for a variational treatment) by making use of different approaches and techniques adapted to the framework of the fully nonlinear (anisotropic) models. Moreover, we also generalize, to some degree, the results by Figalli and Shahgholian \cite{FS} and Indrei and Minne \cite{IM} to the degenerate fully nonlinear context when $\Omega$ (in their setting) equals $\{u>\phi\}$ and $u\geq\phi$ is imposed. A rather interesting open problem is to obtain results as in the cited papers without the constraint $u\geq\phi$.

Furthermore, to the best of the authors' knowledge, the results presented here comprise the first known results of obstacle type problems ruled by degenerate equations in non-divergence form, and they are new even for the simplest operator $\mathcal{G}[u] \defeq |Du|^\gamma\Delta u$.

As an interesting consequence, we obtain the sharp quadratic behavior (along free boundary points) for viscosity solutions to the homogeneous obstacle problem (cf. \cite{L}).

\begin{cor}\label{CorC1,1} Suppose that the assumptions of Theorem \ref{main} are in force with $f \equiv 0$ and $\phi \in C^{1, 1}(B_1)$. Then, for any point $x_0 \in \partial \{u>\phi\} \cap B_{\frac{1}{2}}$ and $0<r\ll1$
$$
\sup_{B_r(x_0)} \frac{u(x)-(u(x_0)+ D u(x_0)\cdot (x-x_0))}{r^{2}}\leq C_{\ast}.\left[\|u\|_{L^{\infty}(B_1)} + \|\phi\|_{C^{1, 1}(B_1)}\right].
$$
Particularly,
$$
\sup_{B_r(x_0)} \frac{u(x)-\phi(x)}{r^{2}}\leq C_{\sharp}.\left[\|u\|_{L^{\infty}(B_1)} + \|\phi\|_{C^{1, 1}(B_1)}\right],
$$
i.e. $u$ leaves the obstacle in a quadratic fashion.
\end{cor}

\begin{rem}[{\bf Sharpness Theorem \ref{main}}]

It is interesting to point out that the degenerate nature of the operator entails specific difficulties that are not present in the uniformly elliptic case. As a matter of fact, our problem essentially different from the problem
\[
  \min\left\{f-|Du|^\gamma F(D^2u),u-\phi\right\}  = 0 \quad \text{in} \quad B_1
\]
studied in \cite{DSV} and, unlike that case, one cannot expect solutions of \eqref{Eq1} to be in $C^{1, 1}$ even when the data is smooth. Indeed, for any fixed $0<r<1$ the radially symmetric function $v: B_1\to \R$ given by
$$
 \,v(x)=(|x|-r)_+^{\,\frac{\gamma+2}{\gamma+1}}
$$
satisfies (in the viscosity sense)
$$
   \left \{
       \begin{array}{rllll}
           |D v|^{\gamma}\Delta v & = & f(x)\chi_{\{v>0\}} & \text{ in } & B_1 \\
           v(x) & \geq & 0 &\text{ in } &  B_1 \\
           v(x) & = & (1-r)^{\,\frac{\gamma+2}{\gamma+1}} &\text{ on } & \partial B_1 \\
           v(x) & = & 0 &\text{ in } &  \overline{B_r}
       \end{array},
   \right.
$$
where
$$
f(x)\defeq\left\{
\begin{array}{rcl}
  \left(\frac{\gamma+2}{\gamma+1}\right)^{\gamma+1}\left(\frac{1}{1+\gamma}+(n-1)\left(1-\frac{r}{|x|}\right)\right) & \text{if} & r<|x|<1 \\
  \frac{(\gamma+2)^{\gamma+1}}{(\gamma+1)^{\gamma+2}} & \text{if} & |x|\leq r
\end{array}
\right.
$$
Note that $v\in C^{1,\frac{1}{\gamma+1}}$, however $v\notin C^{1,\frac{1}{\gamma+1}+\varepsilon}$ for any $\varepsilon>0$, and in particular is not $C^{1,1}$, except in the case when $\gamma = 0$. These features highlight both the sharpness of our results and the intrinsic restrictions poised by the degenerate character of the problem.
\end{rem}

As a consequence of the previous theorem \ref{main}, we also find the sharp rate at which gradient grows away from free boundary points.

\begin{thm}[{\bf Sharp gradient growth}]\label{sharpgrad} Suppose that the assumptions of Theorem \ref{main} are in force. Let $u$ be a bounded viscosity solution to the $(F, \gamma,\phi, f)$-obstacle problem. Then, for any point $x_0 \in \partial \{u > \phi\} \cap B_{1/2}$ there exists a universal constant $C>0$ such that for all $0<r < \frac{1}{4}$
$$
   \displaystyle \sup_{B_r(x_0)} \frac{|D u(x)-Du(x_0)|}{r^{\beta}} \leq  C.\left[\|u\|_{L^{\infty}(B_1)} + \left(\|\phi\|_{C^{1, \alpha}(B_1)}^{\gamma+1}+\|f\|_{L^{\infty}(B_1)}\right)^{\frac{1}{\gamma+1}} \right].
$$

In particular
$$
  \displaystyle  |D u(y)-Du(x_0)| \leq  C.\dist(x_0, \partial \{u>\phi\})^{\beta}\quad\textrm{ for any } y\in B_r(x_0).
$$
\end{thm}

An instrumental (geometric) interpretation to Theorem \ref{main} says the following: if $u$ solves an $(F, \gamma,\phi, f)$-obstacle problem and $x_0 \in \partial \{u>\phi\} \cap \{|Du|\lesssim r^{\beta}\}$ (``singular zone''), then near $x_0$ we obtain
$$
    \displaystyle \sup_{B_r(x_0)} |u(x)|\leq |u(x_0)|+C.r^{1+\beta},
$$

On the other hand, from a (geometric) regularity viewpoint, it is pivotal qualitative information to obtain the  (counterpart) sharp lower estimate. %for operator with ``frozen coefficients''. 
Such a weak geometric property is designated \textit{non-degeneracy} of solutions.

\begin{thm}[{\bf Non-degeneracy estimates}]\label{ThmNonDeg} Suppose that the assumptions of Theorem \ref{main} are in force. Let $u$ be a bounded viscosity solution to the $(F, \gamma,\phi, f)$-obstacle problem fulfilling $\displaystyle \mathfrak{m} \defeq \inf_{B_1} f(x)>0$. Given $x_0 \in \partial\{u>\phi\}$, there exists a constant $\mathfrak{c} = \mathfrak{c}(n, \mathfrak{m}, \lambda, \Lambda, \gamma)>0$, such that
$$
  \displaystyle \sup_{\partial B_r(x_0)} (u(x)-\phi(x_0))\geq \mathfrak{c}.r^{1+\frac{1}{\gamma+1}}\quad \text{for all} \quad  0<r < \frac{1}{2}.
$$
\end{thm}

Furthermore, from a ``free boundary regularity'' perspective it is of use to have control of the decay of the function $u-\phi$ near free boundary points. These kind of result usually require a stronger non-degeneracy property on the obstacle. That is what we achieve in the next Theorem, where we prescribe the obstacle to be a strict super-solution:

\begin{thm}[{\bf Non-degeneracy away from the free boundary}]\label{ThmNonDegHom} Suppose that the assumptions of Theorem \ref{main} are in force. Let $u$ be a bounded viscosity solution to the $(F, \gamma,\phi, 0)$-obstacle problem with obstacle $\phi \in C^{1,1}(B_{1})$. Suppose further that %\eqref{eq.assumption} is in force
\begin{equation}\label{eq.ob}
\displaystyle |D \phi|^{\gamma}F(x,D^2 \phi)\le -\Upsilon \quad \text{in} \quad B_1
\end{equation}
in the viscosity sense for some constant $\Upsilon>0$.

Given $x_0 \in \partial\{u>\phi\}\cap B_{1/2}$, there exists a universal constant $\mathfrak{c}>0$ such that
$$
  \displaystyle \sup_{\partial B_r(x_0)} (u(x)-\phi(x))\geq \mathfrak{c}.r^{2}\quad\text{for all} \quad  0<r < \frac{1}{2}.
$$
\end{thm}

As consequence of the non-degeneracy of Theorem \ref{ThmNonDegHom} we can show that the free boundary has a sort of measure-theoretic properties provided the obstacle satisfies \eqref{eq.ob}. This requires us to recall the definition of \emph{porosity}: a bounded measurable set $E \subset \R^n$ is called $\varsigma$-porous if for any $x\in E$ there exists a $\varsigma\in(0,1)$ such that for any ball $B_r(x)$ there exists $y\in B_r(x)$ such that
\[
B_{\varsigma r}(y)\subset B_r(x)\setminus E.
\]
Notice that if $E$ is $\varsigma$-porous and $x\in E$ then
\[
\frac{|B_r(x)\cap E|}{|B_r(x)|}=\frac{|B_r(x)|-|B_r(x)\setminus E|}{|B_r(x)|}\leq 1-\varsigma^n,
\]
so that $E$ has no points of density one. Moreover, a $\varsigma-$porous set has $\mathcal{H}_{\text{dim}}(E)\leq n-c\varsigma^n$ (Hausdorff dimension), where $c=c(n)>0$ is a dimensional constant; in particular, a $\varsigma-$porous set has Lebesgue measure zero (see, \cite{Zaj} for more details).

Now we can state the following corollary:

\begin{cor}\label{Cor3}
Suppose that the assumptions of Theorem \ref{ThmNonDegHom} are in force with $f\equiv0$. Then, the free boundary is $\varsigma-$porous and there exists  $\epsilon = \epsilon(n, \varsigma)>0$ such that
$$
   \mathcal{H}_{\text{dim}}\left(\partial\{u>\phi\}\cap B_{\frac{1}{2}}\right)\leq n-\epsilon.
$$

In particular, the free boundary has zero Lebesgue measure.
\end{cor}

\subsection{Insights behind the proofs and main difficulties to overcome}

As mentioned before, the main idea of the proof of Theorem \ref{main} is to use an geometric decay argument along those free boundary point around which the equation degenerates, i.e. where the gradient becomes very small.

The first key step in that direction is an Improvement of Flatness Lemma (Lemma \ref{lem.flat} below) that says, roughly speaking, that solutions with small right hand side have to be somewhat $\phi-$flat. This is a powerful device in nonlinear (geometric) regularity theory and plays a pivotal role in our approach. The core idea was inspired in the flatness improvement reasoning from \cite[Lemma 4.2]{daSLR} and references therein. Notwithstanding, the general class of operators (and the novelty scenery with the obstacle constraint) which we are dealing with imposes some significant adjusts to such strategies. As a matter of fact, different from \cite{daSLR} we are not allowed to conclude the proof via a strong maximum principle for the profile $v = u -\phi$. We overcome such an obstacle by invoking a ``Cutting Lemma'' for general fully nonlinear degenerate elliptic equations (Lemma \ref{cutting}), and by showing that such profiles are also viscosity (super)solutions of a fully nonlinear uniformly elliptic equation, which opens up the way to use the Strong Maximum Principle.

After that, the aim is to make use of this flatness argument to ensure that viscosity solutions are ``geometrically close'' to their tangent plane, i.e.
$$
    \displaystyle\sup_{B_{\rho}} \left|u(x)-u(0)-D u(0)\cdot x\right|\leq \rho^{1+\beta},
$$
thereby getting a first geometric estimate (see Lemma \ref{lem.firststep} for further details). Different from linear scenario and second order operator with lower order terms, it is noteworthy that the Lemma \ref{lem.firststep}, which represents the first step in an iteration process, is actually not enough to proceed with a standard iterative scheme as those used in \cite[Theorem 3.1]{ART15} or \cite[Section 5 and 6]{DT17}, i.e.,
$$
\displaystyle\sup_{B_{\rho^k}} \left|u(x)-\mathfrak{l}_k(x)\right|\leq \rho^{k(1+\beta)},
$$
because \textit{a priori} we do not know the equation which is satisfied by
$$
 v_k(x) \defeq \frac{(u-\mathfrak{l}_k)(\rho^kx)}{\rho^{k(1+\beta)}},
$$
where $\{\mathfrak{l}_k\}_{k\in\mathbb{N}}$ is sequence of affine functions (remind that $v \mapsto  |D v|^{\gamma}F(x, D^2 v)$ is not invariant by affine mappings). Nevertheless, it provides the following quantitative information on the oscillation of $u$ inside $B_{\rho}$:
$$
\displaystyle \sup_{B_{\rho}}\left|u(x)-u(0)\right|\leq\rho^{1+\beta}+\rho|D u(0)|,
$$
which proves to be the appropriate estimate to go forward with the iterative procedure, provided we get a control under the magnitude of the gradient in a suitable manner.

We close this introduction by presenting some known results that will be used later on. The next result is the ``Cutting Lemma'' from \cite{IS13} and it is concerned with the homogeneous degenerate problem:

\begin{lem}[{\bf Cutting Lemma}]\label{cutting}
Let $F$ be an operator satisfying \eqref{UnifEllip} and \eqref{eqxdep} and $u$ be a viscosity solution of
\[
|Du|^\gamma F(x,D^2u)=0 \quad \textrm{ in }\quad B_1.
\]
Then $u$ is viscosity solution of
\[
F(x,D^2u)=0 \quad \textrm{ in }\quad B_1.
\]
\end{lem}

The proof follows the exactly the same lines as the one in \cite[Section 5]{IS13} for translation invariant operators once the continuity assumption on the $x$ dependence of $F$, i.e. \eqref{eqxdep}, is in force.

Another important tool is the Comparison Principle:

\begin{thm}[{\bf Comparison Principle, \cite[Theorem 1.1]{BerDem}}]\label{Thm comparison principle} Let $\gamma>0$, $F$ satisfying \eqref{UnifEllip} and \eqref{eqxdep}$, f \in C^0(\overline{B_1})$ and $h$ a continuous increasing function satisfying $h(0)=0$. Suppose $u_1$ and $u_2$ be continuous functions are respectively a viscosity supersolution and subsolution of
\[
|Du|^\gamma F(x, D^2u) = h(u)+f(x) \quad \text{in} \quad B_1.
\]
If $u_1 \geq u_2$ on $\partial B_1$, then $u_1 \geq u_2$ in $B_1$.

Furthermore, if $h$ is nondecreasing (in particular if $h\equiv 0$), the result holds if $u_1$ is a strict supersolution or vice versa if $u_2$ is a strict subsolution. 	
\end{thm}

The next result is a Comparison Principle adjusted to our context; its proof is a consequence of Theorem \ref{Thm comparison principle} (see, \cite[Lemma 2.3]{DSV} for more details).

\begin{lem}[{\bf Comparison Principle}]\label{comparison principle}
Let $u_1$ and $u_2$ be continuous functions in $\Omega$ and $f \in C^0(\overline{B_1})$ fulfilling
$$
    |Du_1|^\gamma F(x, D^2u_1) \leq f(x) \leq |Du_2|^\gamma F(x, D^2u_2) \quad  \text{ in } \quad B_1
$$
in the viscosity sense for some uniformly elliptic operator $F$ and $\displaystyle \inf_{B_1} f>0$ or $\displaystyle \sup_{B_1} f<0$. If $u_1 \geq u_2$ on $\partial B_1$, then $u_1 \geq u_2$ in $B_1$.
\end{lem}

Finally, we present a fundamental gradient estimate for the unconstrained problem, which play a crucial role in obtaining our estimates along free boundary points of solutions. Such a result (stated in a ``non optimal form'') can be found in \cite[Theorem 1]{IS13} (see also \cite[Theorem 3.1]{ART15}).

\begin{thm}[{\bf Gradient estimates}]\label{GradThm} Let $F$ be an operator satisfying \eqref{UnifEllip} and \eqref{eqxdep} and let $u$ be a bounded viscosity solution to
$$
 |Du|^\gamma F(x , D^2u) = f \in L^{\infty}(B_1).
$$
Then,
\begin{equation*}
\|u\|_{C^{1,\sigma}(B_{1/2})}\leq C.\left(\|u\|_{L^{\infty}(B_1)}+ \|f\|_{L^{\infty}(B_1)}^{\frac{1}{\gamma+1}}\right),
\end{equation*}
for universal constants $\sigma<<1$ and $C>0$.
\end{thm}

\section{Proof of Theorems \ref{main} and \ref{sharpgrad}}\label{sec.main}

This section is devoted to the proof of Theorems \ref{main} and \ref{sharpgrad}. As mentioned in the Introduction, the strategy for Theorem \ref{main} is to use a geometric oscillation decay argument to control the growth of solutions along those free boundary points in which the equation degenerates, i.e. those points where the gradient becomes small (in a sense to be made precise later). On the other hand, at those points where the gradient is bounded below the equation is uniformly elliptic and classical estimates can be applied. We start by defining the appropriate localized solutions to our problem:

\begin{defi}\label{FineClass} For a fully nonlinear operator $F$ fulfilling \eqref{UnifEllip} and \eqref{eqxdep}, $\phi\in C^{1,\alpha}(B_1)$ with $\alpha\in(0,1]$ and $f\in L^\infty(B_1)\cap C^0(B_1)$. We say that $u \in \mathfrak{J}(F, \phi, f)(B_1)$ if
\begin{itemize}

  \item $|Du|^{\gamma}F(x, D^2 u) = f(x)\chi_{\{u>\phi\}}$ in $B_1$ in the viscosity sense;
  \item  There exists $L \in (0, \infty)$ such that $-L \leq u\leq L$ in $B_1$ (see, \cite[Theorem 1]{DFQ09});
  \item $u\geq \phi$ in $B_r$ and $u(0)=\phi(0)$.
\end{itemize}
\end{defi}

\begin{rem}[{\bf Normalization assumption}]\label{RemNormAss} Without loss of generality, we will perform our estimates for solutions in these normalized classes and the precise estimates \eqref{EqSRE} and \eqref{detach} follow simply by scaling appropriately. We may further assume that $u$ solves the $(F, \gamma,\phi, f)$-obstacle problem with obstacle $\phi$ and source term $f$ fulfilling
\[
\|u\|_{L^{\infty}(B_1)}\leq 1, \quad \|\phi\|_{C^{1, \alpha}(B_1)} \leq \frac{1}{2} \quad\textrm{ and }\quad \|f\|_{L^{\infty}(B_1)}\leq \delta_0, \]
for any given $\delta_0>0$. As a matter of fact, let us consider the normalized function:
\[
v(x) \defeq \frac{u(x)}{\|u\|_{L^{\infty}(B_1)} + \left(2^{\gamma+1}\|\phi\|_{C^{1, \alpha}(B_1)}^{\gamma+1}+\delta_0^{-1}\|f\|_{L^{\infty}(B_1)}\right)^{\frac{1}{\gamma+1}}}.
\]
$v$ thus defined will satisfy an equation like \eqref{Eq1} with $F,f$ and $\phi$ replaced by $\hat{F},\hat{f}$ and $\hat{\phi}$ respectively where
\[
\left\{
\begin{array}{rcl}
 \kappa & \defeq & \|u\|_{L^{\infty}(B_1)} + \left(2^{\gamma+1}\|\phi\|_{C^{1, \alpha}(B_1)}^{\gamma+1}+\delta_0^{-1}\|f\|_{L^{\infty}(B_1)}\right)^{\frac{1}{\gamma+1}}\\
  \hat{F}(x, X) & \defeq & \kappa^{-1}F\left(x, \kappa X\right) \\
  \hat{f}(x) & \defeq & \kappa^{-(\gamma+1)}f(x) \\
  \hat{\phi}(x) & \defeq & \kappa^{-1}\phi(x).
\end{array}
\right.
\]
Furthermore, $\hat{F}$ is still elliptic with the same ellipticity constants as $F$, and $\hat{f}$ and $\hat{\phi}$ fall into the desired statements.
\end{rem}

The first key step towards the proof of Theorem \ref{main} is Lemma \ref{lem.flat}, which states that solutions with a small enough right hand side are themselves somewhat flatter in the interior. We start however with the following simple stability result that will be instrumental in the proof of Lemma \ref{lem.flat}:

\begin{lem}\label{lem.stab} Let $\{F_k(x,X)\}_{k\in \mathbb{N}}$ be a sequence of operators satisfying \eqref{UnifEllip} with the same ellipticity constants and \eqref{eqxdep} for the same modulus of continuity in $B_1$. Then there exists an elliptic operator $F_0$ which still satisfies \eqref{UnifEllip} and \eqref{eqxdep} such that
\[
F_k\longrightarrow F_0
\]
uniformly con compact subsets of $\text{Sym}(n)\times B_1$.
\end{lem}

\begin{proof}
The proof follows by a standard application of the diagonalization argument of Arzel\`a-Ascoli's Theorem once we note that the ellipticity condition \eqref{UnifEllip} is equivalent to
\[
\lambda\|(X-Y)^+\|-\Lambda\|(X-Y)^-\|\leq F(x, X)-F(x, Y)\leq \Lambda\|(X-Y)^+\|-\lambda\|(X-Y)^-\|
\]
where $A^+$ (resp. $A^-$) stands for the positive (resp. negative) part of the matrix $A$. This implies the uniform Lipschitz character of the sequence $\{F_k\}_{k\in \mathbb{N}}$ on the matrix variable and the result readily follows.
\end{proof}

We are in a position to prove the Improvement of Flatness Lemma:

\begin{lem}[{\bf Improvement of Flatness Lemma}]\label{lem.flat} Given $0<\iota<1$, there exists a $\delta_{\iota}= \delta_{\iota}(\iota,n, \lambda, \Lambda, \gamma)>0$ such that if $u \in \mathfrak{J}(F, \phi, f)(B_{1})$ with
$$
\max\left\{\Theta_F(x, 0), \left\|f\right\|_{L^{\infty}(B_{1})}\right\} \leq \delta_{\iota}
$$
(recall \ref{eqxdep} for the definition of $\Theta_F$). Then,
\begin{equation}\label{eq.flat}
      \max\left\{\|u-\phi\|_{L^{\infty}\left(B_{\frac{1}{2}}\right)}, \|D(u-\phi)\|_{L^{\infty}\left(B_{\frac{1}{2}}\right)}\right\} \leq \iota.
\end{equation}
\end{lem}

\begin{proof} Suppose for sake of contradiction that the thesis of the lemma fails to hold. This means that for some $\iota_0\in (0, 1)$ we can find a sequence $\{u_k\}_k$ satisfying:
\begin{itemize}
\item $u_k\in \mathfrak{J}(F_k, \phi_k, f_k)(B_{1})$;
\item $\max\left\{\Theta_{F_k}(x, 0), \|f_k\|_{L^{\infty}(B_{1})}\right\} \leq \frac{1}{k}$;
\item $\|\phi_k\|_{C^{1, \alpha}(B_1)} \leq \frac{1}{2}$;
\item $|D \phi_k|^{\gamma}F(x, D^2 \phi_k) \leq \|f_k\|_{L^{\infty}(B_1)}$.
\end{itemize}
but
  \begin{equation}\label{Eqcont}
     \max\left\{\|u_k-\phi_k\|_{L^{\infty}\left(B_{\frac{1}{2}}\right)}, \|D(u_k-\phi_k)\|_{L^{\infty}\left(B_{\frac{1}{2}}\right)}\right\}  > \iota_0 \quad \forall\,\,\, k \in \mathbb{N}.
  \end{equation}

Recall that, by definition,
$$
    u_k(0) = \phi_{k}(0) \quad \text{and} \quad  -1\leq u_k\leq1.
$$
Hence, by H\"{o}lder regularity of solutions (see, \cite{BeDe14} and \cite{IS13}), up to a subsequence, $u_k \to u$  uniformly in $\overline{B_r}$. Furthermore, $\phi_{k} \to  \phi$ and $D \phi_{k} \to D\phi$ locally uniformly. Now, from Theorem \ref{GradThm} we can estimate
$$
  \|u_k\|_{C^{1, \sigma}\left(B_{\frac{1}{2}}\right)} \leq C(n, \lambda, \Lambda, \gamma)\left(\|u_k\|_{L^{\infty}(B_{1})}+\|f_k\|^{\frac{1}{\gamma+1}}_{L^{\infty}(B_{1})}\right)
$$
for some $\sigma>0$. Thus, up to a subsequence, $Du_k \to Du$ uniformly in $\overline{B_{\frac{1}{2}}}$. In particular, from \eqref{Eqcont} we conclude that
\begin{equation}\label{eq.cont}
 \max\left\{\|u-\phi\|_{L^{\infty}\left(B_{\frac{1}{2}}\right)}, \|D(u-\phi)\|_{L^{\infty}\left(B_{\frac{1}{2}}\right)}\right\}\geq \iota_0.
\end{equation}

On the other hand, owing to Lemma \ref{lem.stab}, there exists an elliptic operator $F_0$ satisfying \eqref{UnifEllip} and \eqref{eqxdep} (with $\omega \equiv 0$, i.e. $F_0$ has constant coefficients) such that $F_k \to F_0$ locally uniformly in $\text{Sym}(n)$ for all $x \in B_1$ fixed.

Now, by making use of stability results for viscosity solutions (see for example \cite[Corollary 2.7 and Remark 2.8]{BeDe14}), we have that
$$
    |D\phi|^{\gamma}F_0(D^2 \phi) \leq 0 \leq |Du|^{\gamma} F_0(D^2 u) \quad \text{in} \quad B_{\frac{1}{2}}\,\,\, \quad \mbox{and} \quad u(0) = \phi(0)
$$
and from the ``Cutting Lemma \ref{cutting}'' we conclude that
\[
F_0(D^2 \phi) \leq 0 \leq F_0(D^2 u) \quad \text{in} \quad B_{\frac{1}{2}},
\]
which implies that $u-\phi$ is a non-negative viscosity supersolution for the Pucci extremal operator $\mathcal{M}^-_{\frac{\lambda}{n},\Lambda}$\footnote{Notably this fact is not trivial if $\phi$ is not a $C^2$ function. See, \cite[Theorem 5.3]{CC95} for a proof.} in $B_{\frac{1}{2}}$. Finally, the Strong Maximum Principle for uniformly elliptic operators (see, \cite[Proposition 4.9]{CC95}) and the fact that  $u(0) = \phi(0)$ imply that $u\equiv\phi$ in $B_{1/2}$, which is a contradiction with \eqref{eq.cont} and the proof is finished.

\end{proof}

\begin{rem}[{\bf Smallness regime}]\label{SmallRegime} Let us comment on the scaling properties of our problem which allow us to set us under to hypotheses of the Flatness Lemma \ref{lem.flat} in the proof of Main Theorem \ref{main}. To this end, let $u$ be a viscosity solution of
$$
|D u|^\gamma F(x, D^2u) =  f(x)\chi_{\{u>\phi\}}  \textrm{ in } B_1.
$$

From Remark \ref{RemNormAss} we can suppose without loss of generality that $\|f\|_{L^{\infty}(B_1)}\leq \delta_0$, for any given $\delta_0>0$. Thus, it remains to show that $\Theta_F(x, y)$ is small enough. Now let us fix a point $x_0 \in \partial \{u>\phi\} \cap B_{1/2}$. Then, we define $v: B_1 \to \R$ as follows
$$
v(x) = u(\tau x+ x_0)
$$
for a parameter $\tau>0$ to be determined \textit{a posteriori}. Hence, it is easy to check that $v$ satisfies (in the viscosity sense)
$$
|D v|^\gamma F_{\tau, x_0}(x, D^2v) =  f_{\tau, x_0}(x)\chi_{\{v>\phi_{\tau, x_0}\}}  \quad \textrm{ in } \quad B_1,
$$
where
\[
\left\{
\begin{array}{rcl}
  F_{\tau, x_0}(x, X) & \defeq & \tau^{2}F\left(x_0+\tau x, \tau^{-2} X\right) \\
  f_{\tau, x_0}(x) & \defeq & \tau^{\gamma+2}f(x_0+\tau x) \\
  \phi_{\tau, x_0}(x) & \defeq & \phi(x_0+\tau x)\\
  v(0) & \defeq & \phi_{\tau, x_0}(0).
\end{array}
\right.
\]
Now, for given $\iota>0$ (which in our argument will be sufficiently small but fixed) and the corresponding $\delta_{\iota}>0$ in the statement of Flatness Lemma \ref{lem.flat}, choose
$$
 \tau \defeq \min \left\{\frac{1}{4}, \omega^{-1}\left(\frac{\delta_{\iota}}{C_F+1}\right)\right\}.
$$
Hence,
$$
   \Theta_{F_{\tau, x_0}}(x, y) = \sup_{\substack{X \in  Sym(n) \\ X \neq 0}}\frac{|F_{\tau, x_0}(x,X)-F_{\tau, x_0}(y,X)|}{\|X\|} \leq C_F\omega(\tau|x-y|)\leq \delta_{\iota}.
$$
Therefore, with such choice, as well as by using the Remark \ref{RemNormAss}, $v$, $F_{\tau, x_0}$, $f_{\tau, x_0}$ and $\phi_{\tau, x_0}$ fall into the assumptions of Flatness Lemma \ref{lem.flat}.
\end{rem}

\begin{rem}\label{RemarkTransl} By revisiting the proof of previous Flatness Lemma, we can prove that, under the smallness regime, Remark \ref{SmallRegime}, a similar approximation regime holds true for viscosity solutions of
$$
|Du + \overrightarrow{q}|^{\gamma}F(x, D^2 u) = f(x)\chi_{\{u>\phi\}} \quad \text{in} \quad B_1
$$
for any $\overrightarrow{q} \in \R^n$ arbitrary vector. Such a conclusion holds, one more time, by using compactness/stability/Cutting Lemma, as well as a dichotomic analysis similar to the one employed in \cite[Lemma 6]{IS13}, see also \cite[Lemma 5.1]{ART15} and \cite[Proposition 2.9]{BeDe14} for similar reasoning. For this reason, we have decided omit it here.
\end{rem}

The next Lemma establishes the first step of the geometric control on the growth of the gradient:

\begin{lem}\label{lem.firststep}
Suppose that the assumptions of Lemma \ref{lem.flat} are in force. Then, there exists $\rho\in \left(0,\frac{1}{2}\right)$ such that
\begin{equation}\label{Aproxcond}
\displaystyle\sup_{B_{\rho}} \left|u(x)-u(0)-D u(0)\cdot x\right|\leq \rho^{1+\beta}.
\end{equation}
\end{lem}

\begin{proof}
Let $\iota>0$ to be chosen later. From Lemma \ref{lem.flat} we know that there exists $\delta_\iota>0$, such that whenever $\max\left\{\Theta_F(x, 0), \left\|f\right\|_{L^{\infty}(B_{1})}\right\} \leq \delta_{\iota}$, then \eqref{eq.flat} holds. Fixing $\rho\in \left(0,\frac{1}{2}\right)$ and $x\in B_{\rho}$ we compute
{\small{
$$
\left|u(x)-u(0)-D u(0)\cdot x\right| \leq |u(x)-\phi(x)| + |\phi(x)-\phi(0)-D\phi(0)\cdot x|+ |(D\phi-D u)(0)\cdot x|
$$
}}
so that
\[
\sup_{B_\rho}\left|u(x)-u(0)-D u(0)\cdot x\right| \leq \sup_{B_\rho}|\phi(x)-\phi(0)-D\phi(0)\cdot x|+ 2\iota
\]
where $C>0$ is a universal constant provided that $\max\left\{\Theta_F(x, 0), \left\|f\right\|_{L^{\infty}(B_{1})}\right\}\leq\delta_{\iota}$. Now according to the regularity of the obstacle and our normalizing assumption we have
$$
\begin{array}{ccc}
 \displaystyle \sup_{B_\rho}|\phi(x)-\phi(0)-D\phi(0)\cdot x|  & \leq & \frac{1}{2}\rho^{1+\alpha} \\
   & \leq & \frac{1}{2}\rho^{1+\beta},
\end{array}
$$
and hence
\[
\sup_{B_\rho}\left|u(x)-u(0)-D u(0)\cdot x\right| \leq \frac{1}{2}\rho^{1+\beta}+2\iota.
\]

By fixing
\[
\iota\leq \frac{1}{2}\rho^{1+\beta}
\]
we conclude the proof.
\end{proof}

As mentioned in the Introduction, the previous lemma is actually not enough to iterate so we prove the following simple consequence of it that will serve as the correct estimate:
\begin{cor}\label{c3.1}
Under the assumptions of Lemma \ref{lem.firststep} one has
$$
\displaystyle \sup_{B_{\rho}}\left|u(x)-u(0)\right|\leq\rho^{1+\beta}+\rho|D u(0)|,
$$
where $\rho\in \left(0,\frac{1}{2}\right)$.
\end{cor}
\begin{proof}
Using Lemma \ref{lem.firststep} we estimate
{\small{
$$
\begin{array}{rcl}
  \displaystyle \sup_{B_{\rho}} \left|u(x)-u(0)\right| & \leq &  \displaystyle \sup_{B_{\rho}} \left|u(x)-u(0)-D u(0)\cdot x\right|+\sup_{B_{\rho}}|D u(0)\cdot x| \\
   & \leq & \rho^{1+\beta}+\rho|D u(0)|.
\end{array}
$$}}
\end{proof}

In order to obtain a precise control on the influence of the magnitude of the gradient of $u$, we will iterate solutions (by using Corollary \ref{c3.1}) in $\rho$-adic balls. The proof is inspired by some ideas from \cite[Theorem 3.1]{AdaSRT19} and \cite[Theorems 1.1 and 1.3]{APR17}.

\begin{lem}\label{lem.dy} Under the assumptions of Lemma \ref{lem.firststep} one has
\begin{equation}\label{3.3}
\displaystyle\sup_{B_{\rho^k}}\left|u(x)-u(0)\right|\leq\rho^{k(1+\beta)}+|D u (0)|\sum_{j=0}^{k-1}\rho^{k+j\beta},
\end{equation}
where $\rho\in \left(0,\frac{1}{2}\right)$.
\end{lem}

\begin{proof}
The proof will be by an induction argument. The case $k=1$ is precisely the statement of the Corollary \ref{c3.1}. Suppose now that \eqref{3.3} holds for all the values of $l=1,2,\cdots,k$. Our goal is to prove it for $l=k+1$. Define $v_k: B_1 \to \R$ given by
$$
  \displaystyle v_k(x)\defeq \frac{u(\rho^k x)-u(0)}{\rho^{k(1+\beta)}+|D u(0)|\sum\limits_{j=0}^{k-1}\rho^{k+j\beta}}.
$$
Note that
\begin{itemize}
  \item $v_k(0)=0$;
  \item $\|v_k\|_{L^{\infty}(B_1)}\le 1$ by induction hypothesis;
  \item $ D v_k(x)=\frac{\rho^{k} (D u)(\rho^kx)}{\rho^{k(1+\beta)}+|D u (0)|\sum\limits_{j=0}^{k-1}\rho^{k+j\beta}}$.
\end{itemize}
Now, by defining
\begin{itemize}
  \item $F_k(x, X) \defeq \frac{\rho^{2k}}{\rho^{k(1+\beta)}+|D u (0)|\sum\limits_{j=0}^{k-1}\rho^{k+j\beta}} F\left(\rho^k x, \left(\frac{\rho^{2k}}{\rho^{k(1+\beta)}+|D u (0)|\sum\limits_{j=0}^{k-1}\rho^{k+j\beta}}\right)^{-1}X\right)$;
  \item $f_k(x) \defeq \frac{\rho^{k(\gamma+2)} f(\rho^kx)}{\left(\rho^{k(1+\beta)}+|D u (0)|\sum\limits_{j=0}^{k-1}\rho^{k+j\beta} \right)^{\gamma+1}}$;
  \item $\phi_k(x) \defeq \frac{\phi(\rho^k x)-\phi(0)}{\rho^{k(1+\beta)}+|D u(0)|\sum\limits_{j=0}^{k-1}\rho^{k+j\beta}}$,
\end{itemize}
we obtain
$$
   |D v_k|^{\gamma} F_k(x, D^2 v_k) = f_k(x)\chi_{\{v_k>\phi_k\}} \quad \text{in} \quad B_1
$$
in the viscosity sense. Furthermore, it is easy to check that: $F_k$ satisfies \eqref{UnifEllip} and \eqref{eqxdep},
$$
v_k(0) = \phi_k(0), \,\,\,\max\left\{\Theta_{F_k}(x, 0), \|f_k\|_{L^{\infty}(B_{1})}\right\} \ll 1 \,\,\text{and}\,\,
  |D \phi_k|^{\gamma} F_k(x, D^2 \phi_k) \leq \|f_k\|_{L^{\infty}(B_{1})}.
$$
Therefore, $F_k, f_k, \phi_k$ fall into the assumptions of Flatness Lemma \ref{lem.flat}. For this reason, we can apply Corollary \ref{c3.1} to $v_k$ and obtain
$$
\displaystyle \sup_{B_{\rho}} \left|v_k(x)-v_k(0)\right|\leq\rho^{1+\beta}+\rho|D v_k(0)|,
$$
which implies
$$
\displaystyle\sup_{B_{\rho}}\frac{|u(\rho^k x)-u(0)|}{\rho^{k(1+\beta)}+|D u (0)|\sum\limits_{j=0}^{k-1}\rho^{k+j\beta}}\leq \rho^{1+\beta}+\frac{\rho^{k+1}|D u(0)|}{\rho^{k(1+\beta)}+|D u(0)|\sum\limits_{j=0}^{k-1}\rho^{k+j\beta}},
$$
which, by scaling back provides
$$
  \displaystyle\sup_{B_{\rho^{k+1}}}|u(x)-u(0)| \leq  \rho^{(k+1)(1+\beta)}+ |D u(0)|\sum\limits_{j=0}^{k}\rho^{k+1+j\beta},
$$
thereby completing the $(k+1)-$step of induction.
\end{proof}

The next result leads to a sharp regularity estimate in the singular zone.

\begin{lem}\label{l3.3}
Suppose that the assumptions of Lemma \ref{lem.flat} are in force. Then, there exists a universal constant $M>1$ such that, for $\rho$ as in the conclusion of Lemma \ref{lem.flat},
$$
\displaystyle \sup_{B_{r}}|u(x)-u(0)|\leq Mr^{1+\beta}\left(1+|D u(0)|r^{-\beta}\right),\,\,\forall r\in(0,\rho),
$$
where $\rho \left(0,\frac{1}{2}\right)$.
\end{lem}

\begin{proof}
Firstly, fix any $r\in(0,\rho)$ and choose $k\in\mathbb{N}$ such that $\rho^{k+1}<r\leq\rho^{k}$. By using Lemma \ref{lem.dy}, we estimate
\begin{align*}
\sup_{B_r}\frac{|u(x)-u(0)|}{r^{1+\beta}} & \leq \frac{1}{\rho^{1+\beta}} \sup_{B_{\rho^k}}\frac{|u(x)-u(0)|}{\rho^{k(1+\beta)}} \\
										   & \displaystyle \leq \frac{1}{\rho^{1+\beta}} \left(1+|Du(0)|\rho^{-k(1+\beta)} \sum_{j=0}^{k-1}\rho^{k+j\beta} \right)\\
										   & \leq \frac{1}{\rho^{1+\beta}} \left(1+|Du(0)|\rho^{-k\beta}\sum_{j=0}^{k-1}\rho^{j\beta}\right)\\
										   & \leq \frac{1}{\rho^{1+\beta}} \left(1+|Du(0)|\rho^{-k\beta}\frac{1}{1-\rho^\beta}\right)\\
   										   & \leq M(1+|Du(0)|r^{-\beta})\\
\end{align*}
which finishes the proof by choosing $M \defeq \frac{1}{\rho^{1+\beta}(1-\rho^\beta)}>1$.
\end{proof}

Now we can give the proof of the main result of this manuscript, namely Theorem \ref{main}.

\begin{proof}[{\bf Proof of Theorem \ref{main}}]
Without loss of generality, we may assume that $x_0=0$. Notice that the degenerate ellipticity of the operator naturally leads us to separate the study into two different regimes depending on whether $|Du(0)|$ is ``small'' or not (cf. \cite[Theorem 1]{ALS15}).

Firstly, for the radius $r \in (0, \rho)$ fixed in the previous Lemma \ref{lem.dy}, we consider:
\vspace{0.4cm}

{\bf Case 1: $|D u(0)|\leq r^\beta$} (the ``singular zone'')

 By using Lemma \ref{lem.dy} we estimate
\begin{align*}
\sup_{B_r}\left|u(x)-u(0)-D u(0)\cdot x\right| & \leq \sup_{B_r}|u(x)-u(0)|+|D u(0)|r \\
												 &  \leq Mr^{1+\beta}\left(1+|D u(0)|r^{-\beta}\right)+r^{1+ \beta}\\
												 & \leq  3Mr^{1+\beta}\\
\end{align*}
as desired in this case.
\vspace{0.4cm}

{\bf Case 2: $|D u(0)|> r^\beta$} (the set of ``non-singular points'')

In this setting we have
\begin{equation}\label{EqSRE2}
     \displaystyle \sup_{B_{r_{\ast}}(0)} |u(x)-u(0)|\leq Cr_{\ast}^{1+\beta},
\end{equation}
for $0<r_{\ast}(r) < \frac{1}{4}$, where $C>0$ is a universal constant.

In effect, taking the particular case $r_{0} = \sqrt[\beta]{|D u(0)|}$ we are allowed to apply the previous case and conclude that
\begin{equation}\label{EqCzse1}
     \displaystyle \sup_{B_{r_{0}}(0)} |u(x)-u(0)|\leq Cr_{0}^{1+\beta}.
\end{equation}
As before, we define the re-scaled auxiliary functions:
$$
\left\{
\begin{array}{rcl}
u_{r_{0}}(x) & \defeq & \frac{u(r_{0}x)-u(0)}{r_{0}^{1+\beta}}\\
\phi_{r_{0}}(x) & \defeq & \frac{\phi(r_{0}x)-\phi(0)}{r_{0}^{1+\beta}}\\
 f_{r_{0}}(x) & \defeq & r_{0}^{1-\beta(\gamma+1)}f(r_{0}x)\\
 F_{r_0}(x, X)& \defeq & r_0^{1-\beta}F\left(r_0x, \frac{1}{r_0^{1-\beta}}X\right).
\end{array}
\right.
$$
Notice that
\begin{equation}\label{EqEstGrad}
  |Du_{r_{0}}(0)| = |D\phi_{r_{0}}(0)| = 1 \quad \mbox{and} \quad \|f_{r_0}\|_{L^{\infty}(B_1)} \leq 1.
\end{equation}
and $u_{r_{0}}$ is a viscosity solution to
$$
|D u_{r_{0}}|^{\gamma}F_{r_{0}}(x, D^2 u_{r_{0}}) = f_{r_{0}}(x)\chi_{\{u_{r_{0}}> \phi_{r_{0}}\}} \quad \text{in} \quad B_1
$$

Moreover, \eqref{EqCzse1} assures us that $u_{r_{0}}$ is uniformly bounded in the $L^{\infty}\left(B_{\frac{1}{2}}\right)-$topology. Thus, Theorem \ref{GradThm} it follows that
$$
   \|u_{r_{0}}\|_{C^{1, \sigma}\left(B_{\frac{1}{2}}\right)} \leq C_1^{\sharp} \quad (\mbox{universal constant}).
$$
Such an estimate and \eqref{EqEstGrad}, permit us to choice a radius $r_1 \ll 1$ (universal) so that
$$
   \mathfrak{c}\leq |D u_{r_{0}}(x)| \leq \mathfrak{c}^{-1} \quad \forall\,\,\,x \in B_{r_0}(z_0)\,\,\,\mbox{and}\,\,\,\mathfrak{c}\in (0, 1)\,\,\,\mbox{fixed}.
$$
In particular, we obtain (in the viscosity sense)
$$
   \min\left\{\mathrm{K}(x)-F_{r_0}(D^2 u_{r_{0}}), \,\, u_{r_{0}}-\phi_{r_{0}} \right\}=0,
$$
where $\mathrm{K}(x) = \mathrm{K}\left(\mathfrak{c}, \alpha, \gamma, \|f_{r_{0}}\|_{L^{\infty}(B_1)}\right)>0$.

Therefore, $u_{r_{0}}$ is a viscosity solution (uniformly bounded) to an obstacle-type problem for a uniformly elliptic and convex operator in $B_{r_0}(z_0)$ (with $F_{r_0}(x, 0) = 0$ for every $x \in B_1$), with a $C^{1,\alpha}$ obstacle ($\phi_{r_{z_0}, z_0}$) and uniform bounded source term $\mathrm{K}$. From theory for obstacle-type problems (see, \cite[Theorem 5.4]{BLOP})
$$
   \|u_{r_{0}}\|_{C^{1, \beta}\left(B_{r_0}\right)} \leq C(\alpha, \gamma, \Lambda, \lambda, n)
$$
By scaling back we conclude that
$$
   \|u\|_{C^{1, \beta}\left(B_{r_1r_{0}(0)}\right)} \leq C(\alpha, \gamma, \Lambda, \lambda, n),
$$
which implies that
$$
\displaystyle \sup_{B_r(0)} |u(x)-(u(0)+ D u(0)\cdot x)|\leq Cr^{1+\beta},
$$
for all $r<r_1r_{0}$.

Next, we prove the desired estimate when
\begin{equation}\label{EqintervR}
  r \in \left(r_1\sqrt[\beta]{|Du(0)|}, \sqrt[\beta]{|Du(0)|}\right).
\end{equation}
For this end, suppose that $r$ fulfils \eqref{EqintervR}. Hence,
$$
\begin{array}{rcl}
  \displaystyle \sup_{B_r(0)} |u(x)-(u(0)+ D u(0)\cdot x)| & \leq  & \displaystyle \sup_{B_{r_{1}}(0)} |u(x)-(u(0)+ D u(0)\cdot x)| \\
   & \leq  & Cr_{1}^{1+\beta}\\
   & \leq & \frac{C}{r_{1}^{1+\beta}}
   r^{1+\beta}.
\end{array}
$$
Thus, we obtain the desired estimate for all $r \in \left(0, \frac{1}{4}\right)$.

Finally, we obtain \eqref{detach} by computing the following
$$
\begin{array}{rcl}
  \displaystyle \sup_{B_r(x_0)} |u(x)-\phi(x)| & \leq  & \displaystyle \sup_{B_{r}(x_0)} |u(x)-[u(x_0)+ D u(x_0)\cdot (x-x_0)]|\\
  & + & \displaystyle  \sup_{B_{r}(x_0)} |\phi(x)-[\phi(x_0)+ D \phi(x_0)\cdot (x-x_0)]|\\
   & \leq  & (C+1)r^{1+\beta}.
\end{array}
$$
\end{proof}

As mentioned before, with the aid of Theorem \ref{main} we can prove the growth control on the gradient stated in Theorem \ref{sharpgrad}, thus obtaining a finer gradient control to solutions of \eqref{Eq1} near their free boundary points.

\begin{proof}[{\bf Proof of Theorem \ref{sharpgrad}}]  Let $x_0 \in  \partial \{u>\phi\} \cap B_{1/2}$ be an interior free boundary point. Now, we define the scaled auxiliary function $u_{r, x_0}: B_1 \to \R$ by:
$$
  u_{r, x_0}(x) \defeq \frac{u(x_0+rx)-u(x_0)-rx \cdot Du(x_0)}{r^{1+\beta}}.
$$
Now, observe that $u_{r, x_0}$ fulfils in the viscosity sense
$$
  |D u_{r, x_0}+\overrightarrow{q}_{r, x_0}|^{\gamma}F_{r, x_0}(x, D^2 u_{r, x_0}) = f_{r, x_0}(x)\chi_{\{u_{r, x_0}>\phi_{r, x_0}\}} \quad \text{in} \quad B_1,
$$
where
$$
\left\{
\begin{array}{rcl}
  F_{r, x_0}(x, X) & \defeq & r^{1-\beta}F\left(x_0+rx, \frac{1}{r^{1-\beta}}X\right) \\
  f_{r, x_0}(x)  & \defeq & r^{1-(\gamma+1)\beta}f(x_0 + r x)\\
  \phi_{r, x_0}(x) & \defeq & \frac{\phi(x_0+rx)-\phi(x_0)-rx \cdot D\phi(x_0)}{r^{1+\beta}}\\
  \overrightarrow{q}_{r, x_0} & \defeq & r^{-\beta}Du(x_0).
\end{array}
\right.
$$
From Remark \ref{RemarkTransl} and Theorem \ref{main} we get that
$$
\|u_{r, x_0}\|_{L^{\infty}\left(\frac{1}{4}\right)} \leq C.\left[\|u\|_{L^{\infty}(B_1)} + \left(\|\phi\|_{C^{1, \beta}(B_1)}^{\gamma+1}+\|f\|_{L^{\infty}(B_1)}\right)^{\frac{1}{\gamma+1}}\right].
$$
Finally, by invoking the gradient estimates (Theorem \ref{GradThm}) we obtain that
$$
\begin{array}{rcl}
  \displaystyle  \frac{1}{r^{\beta}} \sup_{B_{\frac{r}{8}}(x_0)} |D u(x)-Du(x_0)| & = & \displaystyle \sup_{B_{\frac{1}{8}}(x_0)} |D u_{r, x_0}(y)| \\
   & \leq  & C\left(n, \gamma, \lambda, \Lambda\right).\left(\|u_{r, x_0}\|_{L^{\infty}\left(\frac{1}{4}\right)} + \|f_{r, x_0}\|_{L^{\infty}\left(B_{\frac{1}{4}}\right)}^{\frac{1}{\gamma+1}}\right) \\
   & \le & C_0.\left[\|u\|_{L^{\infty}(B_1)} + \left(\|\phi\|_{C^{1, \beta}(B_1)}^{\gamma+1}+\|f\|_{L^{\infty}(B_1)}\right)^{\frac{1}{\gamma+1}}\right],
\end{array}
$$
thereby yielding the desired estimate.

For the second part of the Theorem, given $y \in \{u>\phi\}\cap B_{1/2}$, let us pick $z \in \partial (\{u>\phi\}\cap B_{1/2}) = \mathcal{B}$ such that
$$
r_0 \defeq |y-z| = \text{dist}(y, \mathcal{B}).
$$
Now, by using the previous estimates we have
$$
\begin{array}{rcl}
  \displaystyle \sup_{B_{r_0}(y)} |Du(x)-Du(x_0)| & \leq  & \displaystyle \sup_{B_{zr_0}(z)} |Du(x)-Du(x_0)|\\
    & \le & C.(2r_0)^{\frac{1}{\gamma+1}} \\
   & \le & C_0.\text{dist}(y, \mathcal{B})^{\frac{1}{\gamma+1}},
\end{array}
$$
which finishes the proof.
\end{proof}

\section{Non-degeneracy results}\label{sec.nondeg}

This Section is devoted to prove some geometric non-degeneracy properties that play an essential role in the description of solutions to free boundary problems of obstacle type.

\begin{proof}[{\bf Proof of Theorem \ref{ThmNonDeg}}]
Notice that, due to the continuity of solutions, it is sufficient to prove that such a estimate is satisfied just at point within $\{u>\phi\} \cap B_{1/2}$ and the estimate at $\partial{\{u>\phi\}} \cap B_{1/2}$ is obtained by a limiting procedure.

First of all, for $x_0 \in \{u>\phi\} \cap B_{1/2}$ let us define the scaled function
$$
    u_{r}(x) \defeq \frac{u(x_0+rx)}{r^{\frac{\gamma+2}{\gamma+1}}} \quad  \mbox{for} \quad  x \in B_1.
$$

Now, let us introduce the comparison function:
$$
   \displaystyle \Xi (x) \defeq \left\{\frac{ \mathfrak{m} \left(\gamma+1 \right)^{\gamma + 2} }{\left[n(\gamma+1)\Lambda\right]\left( \gamma+2\right)^{\gamma + 1}}\right\}^{\frac{1}{\gamma+1}} |x|^{\frac{\gamma+2}{\gamma+1}} + \frac{1}{r^{\frac{\gamma+2}{\gamma+1}}}\phi(x_0).
$$

Straightforward calculus shows that
$$
     | D \Xi |^\gamma\mathcal{G}_r(x, D^2 \Xi)\leq f_r\left( x  \right) \quad \text{in} \quad B_1
$$
and
$$
   | D u_{r}|^\gamma\mathcal{G}_r(x, D^2 u_{r})= f_r\left( x  \right)  \quad \text{in} \quad B_1\cap \{u_{r}> \phi_{r}\}
$$
in the viscosity sense, where
$$
\left\{
\begin{array}{rcl}
 \mathcal{G}_r(x , X) & \defeq  & r^{\frac{\gamma}{\gamma+1}}F\left(x_0+rx, r^{-\frac{\gamma}{\gamma+1}}X\right) \\
  f_r(x) & \defeq  & f(x_0 + rx) \\
  \phi_{r}(x) & \defeq  & \frac{\phi(x_0+rx)}{r^{\frac{\gamma+2}{\gamma+1}}}.
\end{array}
\right.
$$
Moreover, $\mathcal{G}_r$ satisfies the structural assumptions \eqref{UnifEllip} and \eqref{eqxdep}, and $\displaystyle \inf_{B_1} f_r \geq \inf_{B_1} f>0$.

Finally, if $u_{r} \leq \Xi$ on the whole boundary of $B_1 \cap \{u_{r}> \phi_{r}\}$, then the Comparison Principle (Lemma \ref{comparison principle}), would imply that
$$
   u_{r} \leq \Xi \quad \mbox{in} \quad B_1 \cap \{u_{r}> \phi_{r}\},
$$
which clearly contradicts the assumption that $u_{r}(0)>\phi_{r}(0)$. Therefore, there exists a point $Y \in \partial (B_1 \cap \{u_{r}> \phi_{r}\})$ such that
$$
      u_{r}(Y) > \Xi(Y) = \left\{\frac{ \mathfrak{m} \left(\gamma+1 \right)^{\gamma + 2} }{\left[n(\gamma+1)\Lambda\right]\left( \gamma+2\right)^{\gamma + 1}}\right\}^{\frac{1}{\gamma+1}}
$$
and scaling back we finish the proof of the Theorem.
\end{proof}

Next we will prove our second non-degeneracy result.

\begin{proof}[{\bf Proof of Theorem \ref{ThmNonDegHom}}]

By continuity it is enough to prove the result inside the set where $u$ and $\phi$ are detached. Let then $y \in \{u > \phi\}\cap B_{1/2}$ and $v(x) \defeq \phi(x) +\epsilon|x-y|^{1+\beta}$, where $0<\epsilon \ll 1$ is chosen such that $\displaystyle |Dv|^{\gamma}F(x, D^2 v)<0$ in the viscosity sense.

Now, by putting $r<\mbox{dist}(x_0, \partial B_{1/2})$, we obtain that
$$
   \displaystyle  |Dv|^{\gamma}F(x, D^2 v)<0 \le |Du|^{\gamma}F(x, D^2 u) \quad \mbox{in} \quad \{u>\phi\}\cap B_r(x_0)
$$
in the viscosity sense. Furthermore, $u(y) \geq \phi(y) = v(y)$. By invoking the Comparison Principle (Theorem \ref{Thm comparison principle}) it follows that there is $z_y\in \partial(\{u>\phi\}\cap B_r(x_0))$ such that $u(z_y) \geq v(z_y)$. Since $u <v$ on $B_r(x_0) \cap \partial \{u >\phi \}$ it must hold that $z_y \in \{u > \phi\} \cap \partial B_r(x_0)$ We conclude the proof by letting $y \to x_0$.
\end{proof}

As mentioned before, the porosity of the free boundary is a consequence of the non-degeneracy in the homogeneous case:

\begin{proof}[{\bf Proof of Corollary \ref{Cor3}}]
Let $x_0\in\partial\{u>\phi\}\cap B_{1/2}$ and pick $r$ small enough so that $B_{2r}(x_0)\subset\subset B_{1/2}$. By Theorem \ref{ThmNonDegHom} we have that there exists some $y\in\partial B_r(x_0)$ such that
\begin{equation}\label{porosity1}
u(y)-\phi(y)\geq c.r^{2}
\end{equation}
for some (universal) constant $c>0$.

On the other hand, the growth control proved in Corollary \ref{CorC1,1} gives
\begin{equation}\label{porosity2}
u(y)-\phi(y)\leq C.\big(\textrm{dist}(y,\partial\{u>\phi\})\big)^{2}.
\end{equation}

\eqref{porosity1} and \eqref{porosity2} together imply
\begin{equation}\label{porosity3}
\textrm{dist}(y,\partial\{u>\phi\})>{\bf C}.r
\end{equation}
and taking $\delta\defeq\frac{{\bf C}}{4}$ we obtain that $B_{2\delta r}(y)\cap B_{2r}(x_0)\subset \{u>\phi\}\cap B_{\frac{1}{2}}$ and the result is proved.
\end{proof}

\section{Some examples and extensions}\label{Examples}

In the sequel, we will present some examples where our results hold.

\begin{example} An interesting application of our results when $F(x, X)=\tr(X)$ and $f\equiv 1$. In that case, we get the seemingly simple problem
\[
  \left\{ \begin{array}{rcll}
  |Du|^\gamma\Delta u & = & \chi_{\{u>\phi\}} & \textrm{ in } B_1\\
     u & \geq  & \phi & \textrm{ in } B_1\\
     u & = & g & \textrm{ on } \partial B_1.
\end{array}\right.
\]
To the best of the authors' knowledge, no results whatsoever were available for this toy model. According to Theorem \ref{main}, if $\phi \in C^{1, \frac{1}{\gamma+1}}(\Omega)$ our results give $C^{1,\frac{1}{\gamma+1}}$ regularity, which is the optimal regularity of the unconstrained problem, see for instance \cite[Corollary 3.2]{ART15}, \cite[Example 1]{IS13} and the references therein. Moreover, the results hold true for a general bounded source term $f$.
\end{example}

\begin{example} Our results also hold for Pucci's extremal operators (see the Introduction):
$$
F(D^2u)\defeq \mathcal{M}_{\lambda, \Lambda}^{\pm}(D^2 u)
$$
and, more generally, cover Belmann's type equations, which appear in stochastic control problem as an optimal cost:
\[
\displaystyle F(D^2 u) \defeq \inf_{\hat{\alpha} \in \mathcal{A}}\left(\mathcal{L}^{\hat{\alpha}} u(x)\right) \quad \left(\text{resp.}\,\,\,\sup_{\hat{\alpha} \in \mathcal{B}}  \left(\mathcal{L}^{\hat{\alpha}} u(x)\right)\right),
\]
   and
      $$
         \displaystyle \mathcal{L}^{\hat{\alpha}} u(x) = \sum_{i, j=1}^{n} a_{ij}^{\hat{\alpha}}\partial_{ij} u(x)
      $$
is a family of uniformly elliptic translation invariant operators with ellipticity constants $\lambda$ and $\Lambda$.
\end{example}
\begin{example} Our results also hold for uniform elliptic operators with small ellipticity aperture: For such a class, interior local $C^{2,\alpha}$ \textit{a priori} estimates for solutions of fully nonlinear equations hold under the assumption that the ellipticity constants $(\lambda, \Lambda)$ do not deviate much, in the sense that $\mathfrak{e} \defeq 1 -\frac{\lambda}{\Lambda}$ is small enough (depending only on
dimension). Precisely, if the small ellipticity aperture is in force, then viscosity solutions to the equation
$$
   F(D^2 u) = 0 \quad \text{in} \quad B_1
$$
are $C_{\text{loc}}^{2,\alpha_0}(B_1)$ for some $\alpha_0 \in (0, 1)$. Furthermore, the following estimate holds
$$
   \|u\|_{C^{2,\alpha_0}\left(B_{\frac{1}{2}}\right)} \leq C.\|u\|_{L^{\infty}(B_1)},
$$
where $C>0$ and $\alpha_0$ depend only on dimension and the ellipticity constants.

This result is a consequence of a classical estimate by Cordes and Nirenberg. We recommend the interested reader to consult \cite[Chapter 5]{daS15} for a proof of this result as stated.

Of particular interest, such a result covers Isaac's type equations, which appear in stochastic control and in the theory of differential games:
      $$
      \displaystyle F(x, D^2 u) \defeq \sup_{\hat{\beta} \in \mathcal{B}} \inf_{\hat{\alpha} \in \mathcal{A}}\left(\mathcal{L}^{\hat{\alpha}\hat{\beta}} u(x)\right) \quad \left(\text{resp.}\,\,\,\inf_{\hat{\beta} \in \mathcal{B}} \sup_{\hat{\alpha} \in \mathcal{A}} \left(\mathcal{L}^{\hat{\alpha}\hat{\beta}} u(x)\right)\right),
      $$
      where
      $$
         \displaystyle \mathcal{L}^{\hat{\alpha}\hat{\beta}} u(x) = \sum_{i, j=1}^{n} a_{ij}^{\hat{\alpha}\hat{\beta}}(x)\partial_{ij} u(x)
      $$
      is a family of uniformly elliptic operators with H\"{o}lder continuous coefficients and ellipticity constants $\lambda$ and $\Lambda$ satisfying $\left|1 -\frac{\lambda}{\Lambda}\right| \ll 1$.
\end{example}

\begin{example} Recently, \cite[Theorem 1]{DD19} (see also \cite[Theorem 2]{DD19}) established local $C^{2, \alpha}$ \textit{a priori} estimates (in effect, Schauder type estimates to non-convex fully nonlinear operators) for flat viscosity solutions, i.e., solutions whose oscillations is very small, provided $F \in C^{1, \tau}(\text{Sym}(n))$ and has Dini continuous coefficients. Therefore, such a family of solutions and operators are an interesting class where our results hold true.
\end{example}

Regarding the hypothesis of Theorem \ref{main}, we can actually relax the convexity (or concavity) assumption on the nonlinearity $F$. To this purpose, the key ingredient is an available $C^{1, \alpha}$ (for any $\alpha \in (0, 1)$) regularity theory to
$$
  F(D^2u) = 0 \quad \text{in} \quad \Omega.
$$

Recently, Silvestre and Teixeira in \cite[Theorems 1.1 and 1.4]{ST15} addressed local $C^{1, \alpha}$ regularity estimates to problems with no convex/concave structure. In their approach, the novelty with respect to the former results is the concept of \textit{recession} function (the tangent profile for $F$ at ``infinity'') given by
$$
  \displaystyle F^{\ast}(X) \defeq  \lim_{\tau \to 0+ } \tau F\left(\frac{1}{\tau}X\right).
$$
In this direction, the authors relaxed the hypothesis of $C^{1,1}$ \textit{a priori} estimates for solutions of the equations without dependence on $x$, by the hypothesis that $F$ is assumed to be ``convex or concave'' only at the ends of $\textit{Sym}(n)$, in other words, when $\|D^2 u\| \approx \infty$.
Precisely, they proved that if solutions to the homogeneous equation
$$
F^{\ast}(D^2 u) = 0  \quad \textrm{in} \quad B_1
$$
has $C^{1, \alpha_0}$ \textit{a priori} estimates (for some $\alpha_0 \in (0, 1]$), then viscosity solutions to
$$
 	F(D^2 u) = 0  \quad \textrm{in} \quad B_1 \quad (\text{resp.}  = f \in L^{\infty})
$$
are of class $C^{1, \hat{\alpha}}_{\text{loc}}$ for $\hat{\alpha} <\min\{1, \alpha_0\}$.  	

In conclusion, if the recession profile associated to $F$, which we are calling ``$F^{\ast}$'', enjoys $C^{1,1}$ \textit{a priori} estimates, then a ``good regularity theory'' is available to solutions of $F(D^2u) = 0$. For this reason, we are able to prove our results to operators under either relaxed or no convexity assumptions on $F$.

\begin{example}
As commented above our results hold for operators whose recession profile enjoys appropriate \textit{a priori} estimates (see, one more time, \cite[Theorems 1.1 and 1.4]{ST15}). For the sake of illustration, we will exhibit some operators and its recession counterpart.
Let $$0 < \sigma_1, \ldots , \sigma_n< \infty.$$
We have the following examples:
\begin{enumerate}
  \item[(E1)]({\bf $m$-momentum type operators})
Let $m$ be an odd number. \textit{The $m$-momentum type operator} given by
  $$
\displaystyle  F_m(D^2 u) = F_m(e_1(D^2 u), \cdots, e_n(D^2 u)) \defeq \sum_{j=1}^{n} \sqrt[m]{\sigma_j^m+e_j(D^2 u)^m}-\sum_{j=1}^{n} \sigma_j
  $$
defines a uniformly elliptic operator which is neither concave nor convex. Moreover,
$$
  \displaystyle F_m^{\ast}(X) = \sum_{j=1}^{n} e_j(X)
$$
the Laplacian operator.

\item[(E2)]({\bf Perturbation of ``non-isotropic'' Pucci's operators})
Let us consider
  $$
 F(D^2 u) = F(e_1(D^2 u), \cdots, e_n(D^2 u)) \defeq \sum_{j=1}^{n} \left[h(\sigma_j)e_j(D^2 u) + g(e_j(D^2 u))\right],
  $$
  where $h:[0, \infty) \to \R$ is a continuous function with $h(0) = 0$ and $g: \R \to \R$ is any Lipschitz function such that $g(0) = 0$. Notice that $F$ is uniformly elliptic operator.
  Moreover,
  $$
  \displaystyle F^{\ast}(X) = \sum_{j=1}^{n} h(\sigma_j)e_j(X),
  $$
which is, up a change of coordinates, the Laplacian operator.

\item[(E3)]({\bf Perturbation of the Special Lagrangian equation})
Given $h:\R_{+} \to \R$ a continuous function, the \textit{``perturbation'' of the Special Lagrangian equation}
$$
 F(D^2 u) =  F(e_1(D^2 u), \cdots, e_n(D^2 u)) \defeq \sum_{j=1}^{n} \left[h(\sigma_j)e_j(D^2 u) + \arctan(e_j(D^2 u))\right]
$$
defines a uniformly elliptic operator which is neither concave nor convex. Furthermore,
  $$
  \displaystyle F^{\ast}(X) = \sum_{j=1}^{n} h(\sigma_j)e_j(X),
  $$
which is precisely a ``perturbation'' of the Laplace operator.
\end{enumerate}
\end{example}

\begin{example}[{\cite[Example 3.7]{DSV}}]
When $F: \textit{Sym}(n) \to \R$ is smooth, the \textit{recession profile} $F^{\ast}$ should be understood as the ``limiting equation'' for the natural scaling on $F$. By way of illustration, for a number of operators, it is possible to check the existence of the limit
$$
  \mathfrak{A}_{ij} \defeq \lim_{\|X\|\to \infty} \frac{\partial F}{X_{ij}}(X).
$$
In this situation, $F^{\ast}(X) = \tr(\mathfrak{A}_{ij}X)$. As an example we can consider the operator:
$$
\displaystyle  F_m(e_1(D^2 u), \cdots, e_n(D^2 u)) \defeq \sum_{j=1}^{n} \sqrt[m]{1+e_j(D^2 u)^m}- n,
$$
where $m \in \mathbb{N}$ (an odd number). In such a case, $\displaystyle F^{\ast}(X) = \sum_{j=1}^{n} e_j(X)$ (the Laplacian).
\end{example}

We would like highlight that other interesting class of degenerate operators where our results work out is given by the \textit{$p-$Laplacian} (in its non-divergence form) with $\gamma = p-2$ (for $p > 2$)
     $$
     G_p(\xi, X) \defeq |\xi|^{\gamma}F_p(\xi, X)
     $$
where
$$
F_p(\xi, X)  \defeq \tr\left[\left(I+(p-2)\frac{\xi\otimes \xi}{|\xi|^2}\right)X\right] \quad \text{the Normalized p-Laplacian operator}.
$$
Moreover, for arbitrary $\nu \in \R^n$ such that $|\nu| = 1$ we have
$$
\displaystyle \left\langle I + (p-2)\frac{Du \otimes Du}{|Du|^2}\nu, \nu\right\rangle = \displaystyle 1 + (p-2)\frac{\langle \nu, Du \rangle^2 }{|Du|^2}.
$$
Therefore, we conclude that $F_p$ satisfies \eqref{UnifEllip} with $\lambda = \min\{p-1, 1\}$ and $\Lambda = \max\{p-1, 1\}$.

In its variational form, namely
$$
\Delta_p u = \text{div}(|\nabla u|^{p-2}\nabla u)
$$
such an operator appears for instance as the Euler-Lagrangian equation associated to minimizers of the $p-$energy functional:
$$
\left(g+W_0^{1,p}(\Omega), L^q(\Omega)\right) \ni (w, f) \mapsto \min \mathcal{J}_p(w) \,\,\, \text{with} \,\,\, \mathcal{J}_p(w) = \int_{\Omega}\left(\frac{1}{p}|\nabla w|^p + fw\right)dx.
$$

Problems governed by the $p-$Laplace operator have attracted a huge deal of attention in the last five decades  or so. It is also worth highlighting the series of fundamental works \cite{ATU18} and \cite{APR17}, where the authors address sharp regularity estimates to inhomogeneous problem
$$
-\Delta_p u  = f \in L^{\infty}(B_1) \quad (\text{and its normalized version}).
$$

Finally, by combining the qualitative results from \cite{APR17} (and references therein) with quantitative ones (regularity estimates) from \cite[Theorem 2]{ATU18} we obtain the following result (cf. \cite[Theorem 1]{ALS15} for an alternative approach):
\begin{thm} Let $u$ be a bounded viscosity solution to the obstacle type problem
\begin{equation*}
\left\{
\begin{array}{rcll}
  G_p(Du, D^2 u) & = & f(x)\chi_{\{u>\phi\}} & \textrm{ in } B_1 \\
  	u(x) & \geq & \phi(x) & \textrm{ in } B_1 \\
  u(x) & = & g(x) & \textrm{ on } \partial B_1,
\end{array}
\right.
\end{equation*}
with obstacle $\phi \in C^{1, \alpha}(B_1)$ and $f \in L^\infty(B_1)$. Then, $u \in C_{\text{loc}}^{1, \min\left\{\alpha, \frac{1}{p-1}\right\}}$, More precisely, for any point $x_0 \in \partial \{u>\phi\} \cap B_{\frac{1}{2}}$ there holds
{\small{
$$
       \displaystyle \sup_{B_r(x_0)} \frac{|u(x)-(u(x_0)+ D u(x_0)\cdot (x-x_0))|}{r^{1+\min\left\{\alpha, \frac{1}{p-1}\right\}}}\leq C.\left[\|u\|_{L^{\infty}(B_1)} + \left(\|\phi\|_{C^{1, \alpha}(B_1)}^{p-1}+\|f\|_{L^{\infty}(B_1)}\right)^{\frac{1}{p-1}}\right],
$$
}}
for $0<r < \frac{1}{2}$ where $C>0$ is a universal constant.
\end{thm}

\section*{Appendix}

In spite of the fact that some existence results are available for fully nonlinear uniform elliptic obstacle type  problems (see e.g., \cite{BLOP} and \cite{L}), we could not find in the specialized literature corresponding references to degenerate one (cf. \cite[Theorem 1.1]{DSV}).

For this reason, although is not the main focus of our paper, in this appendix we will address existence/uniqueness and \textit{a priori} estimates of viscosity solutions to degenerate obstacle type problems as follows
\begin{equation}\label{Eq4.1}
  \left\{
\begin{array}{rcll}
  |D u|^\gamma F(x, D^2u)& = & f(x)\chi_{\{u>\phi\}} & \textrm{ in } B_1 \\
  	u(x) & \geq & \phi(x) & \textrm{ in } B_1 \\
  u(x) & = & g(x) & \textrm{ on } \partial B_1,
\end{array}
\right.
\end{equation}
where the fully nonlinear operator $F:B_1\times\text{Sym}(n) \longrightarrow \R$ is supposed to be uniformly elliptic fulfilling \eqref{UnifEllip} and \eqref{eqxdep}.

\begin{thm}[{\bf Existence of solutions with \textit{a priori} estimates}]\label{Thm1} Let $\phi \in C^{1, \alpha}(B_1)$, $f \in C^0(\overline{B_1})$ and $g\in C^{1,\varsigma}(\partial B_1)$ be for some $\alpha, \zeta\in(0,1]$. Suppose further that
$$
   \displaystyle \inf_{B_1} f>0 \,\,\,\text{or}\,\,\, \displaystyle \sup_{B_1} f<0 \,\,\, \text{and} \,\,\,|D \phi|^\gamma F(x, D^2 \phi) \geq 0 \quad \text{in}\,\,\,B_1 \,\,\,(\text{in the viscosity sense}).
$$

Then, there exists a viscosity solution $u$ to the obstacle type problem \eqref{Eq4.1}. Moreover, $u \in C^{1,\vartheta}(\overline{B_1})$ for some universal $\vartheta \in (0, 1)$ and the following \textit{a priori} estimate holds true
\[
   \displaystyle \|u\|_{C^{1,\vartheta}(\overline{B_1})} \leq C(\vartheta).\left(\|u\|_{L^{\infty}(B_1)} + \|g\|_{C^{1, \zeta}(\partial B_1)}+\|f\|_{L^{\infty}(B_1)}^{\frac{1}{\gamma+1}}\right).
\]
\end{thm}

It is worth to highlight that the core ideas behind the proof were inspired/adapted in the ones from \cite[Theorem 3.3]{BLOP}, \cite[Theorem 1.1]{DSV} and \cite[Ch.1, \S 1.3.2]{PSU} and follow by rather standard and well known methods. For the reader's convenience we will present a somewhat complete sketch of the proof.

\begin{proof}
First, we consider the penalized problem
\begin{equation}\label{Eqep}
\left\{
\begin{array}{rcll}
  |D u_{\varepsilon}|^\gamma F(x, D^2u_{\varepsilon})& = & f(x)\Phi_\varepsilon(u_{\varepsilon}-\phi) \pm \varepsilon& \textrm{ in } B_1 \\
  u_{\varepsilon} & = & g & \textrm{ on } \partial B_1.
\end{array}
\right.
\end{equation}
with $\Phi_\varepsilon$ a smooth approximation of the Heaviside function, that is, take $\Phi:\R\longrightarrow[0,1]$ a monotone nondecreasing function satisfying
\[
   \Phi\in C^\infty(\R),\quad\Phi(s)=0\text{ for }s\leq 0,\quad\Phi(s)=1\text{ for }s\geq 1
\]
and define for $\varepsilon\in(0, 1),\: \Phi_\varepsilon(s) \defeq \Phi\left(\frac{s}{\varepsilon}\right)$.

Next we claim that the above problem has a viscosity solution that enjoys a basic regularity estimate (which is not sharp). Indeed, according to Perron's Method\footnote{Such an existence method has as a key ingredient the Comparison Principle tool (Theorem \ref{Thm comparison principle}). For this reason, we stress that we are under the assumptions of Lemma \ref{comparison principle} in order to ensure existence of solutions to auxiliary problem with $f_{\varepsilon}(x) \defeq f(x)\Phi_\varepsilon(v_0-\phi) + \varepsilon$ provided $\displaystyle \inf_{B_1} f>0$ and $f_{\varepsilon}(x) \defeq f(x)\Phi_\varepsilon(v_0-\phi) - \varepsilon$ provided $\displaystyle \sup_{B_1} f<0$.} and \cite[Theorem 1.1]{BeDe14}, it follows that for each $v_0 \in C^0(\overline{B_1})$, there exist $v \in C^{1, \vartheta}(\overline{B_1})$ (for a universal $\vartheta \in (0, 1)$) fulfilling in the viscosity sense
$$
\left\{
\begin{array}{rcll}
  |D v|^\gamma F(x, D^2 v)& = & f(x)\Phi_\varepsilon(v_0-\phi) \pm \varepsilon& \textrm{ in } B_1 \\
  v & = & g & \textrm{ on } \partial B_1.
\end{array}
\right.
$$
such that
\begin{equation}\label{EqEstUptoBound}
  \|v\|_{C^{1, \vartheta}(\overline{B_1})} \leq C(\vartheta).\left(\|v\|_{L^{\infty}(B_1)} + \|g\|_{C^{1, \zeta}(\partial B_1)}+\|f\Phi_{\varepsilon}(\cdot) +1\|_{L^{\infty}(B_1)}^{\frac{1}{\gamma+1}}\right)
\end{equation}

Next, since $\Phi_{\varepsilon}(\cdot) \in [0, 1]$, we can conclude by using \eqref{EqEstUptoBound} and Alexandroff-Bakelman-Pucci estimates (ABP for short) from \cite[Theorem 1]{DFQ09} that
$$
    \|v\|_{C^{1, \vartheta}(\overline{B_1})} \leq C\left(\vartheta, \|f\|_{L^{\infty}(B_1)}, \|g\|_{C^{1, \zeta}(\partial B_1)}\right),
$$
where $C$ does not depends on $v_0$.

In this point, by defining $\mathrm{T}: C^{1, \vartheta}(\overline{B_1}) \to C^{1, \vartheta}(\overline{B_1})$ given by $\mathrm{T}(v_0) = v$, we conclude that $\mathrm{T}$ maps the $C-$ball into itself and it is compact. Therefore, by Schauder's fixed point theorem, there exists $u_{\varepsilon}$ such that $\mathrm{T}(u_{\varepsilon}) = u_{\varepsilon}$, which is a viscosity solution to \eqref{Eqep}.

Now, one more time from \eqref{EqEstUptoBound} and ABP estimates we obtain
$$
\begin{array}{rcl}
  \|u_{\varepsilon}\|_{C^{1, \vartheta}(\overline{B_1})} & \leq & C.\left( \|u_{\varepsilon}\|_{L^{\infty}(B_1)} + \|g\|_{C^{1, \zeta}(\partial B_1)}+\|f.\Phi_{\varepsilon}(u_{\varepsilon}-\phi) +\varepsilon\|_{L^{\infty}(B_1)}^{\frac{1}{\gamma+1}}\right) \\
   & \leq & C.\left[\|g\|_{C^{1, \zeta}(\partial B_1)}+(\|f\|_{L^{\infty}(B_1)}+1)^{\frac{1}{\gamma+1}}\right] \\
   & \leq & C\left(\vartheta, \|f\|_{L^{\infty}(B_1)}, \|g\|_{C^{1, \zeta}(\partial B_1)}\right) \quad (\text{independent on}\,\,\, \varepsilon).
\end{array}
$$
Hence, the family $\{u_{\varepsilon}\}_{\varepsilon \in (0, 1)}$ of solutions to \eqref{Eqep} is uniformly bounded in $C^{1, \vartheta}(\overline{B_1})$. Therefore, Arzel\`{a}-Ascoli's compactness criterium ensures us the existence of a function $u \in C^{1, \vartheta}(\overline{B_1})$ and a subsequence $\{u_{\varepsilon_j}\}_{j \in \mathbb{N}}$ such that
\[
  u_{\varepsilon_j}\longrightarrow u\textrm{ in } C^{1,\vartheta'}(\overline{B_1}) \quad \mbox{for any}\,\,\, \vartheta^{\prime}<\vartheta.
\]

It remains to show that $u$ is a viscosity solution of the fully nonlinear degenerate obstacle problem \eqref{Eq4.1}. Note that $u = g$ on $\partial B_1$ since $\left. u_{\varepsilon_j}\right|_{\partial B_1} =  g$ and $u_{\varepsilon_j} \longrightarrow u$ uniformly in $\overline{B_1}$.

Now, we show that $u \geq \phi$ in $B_1$. In effect, by the uniform convergence of the $u_{\varepsilon_j}$, for each $k \in \mathbb{N}$ fixed, we have that
$$
  f(x)\Phi_{\varepsilon_j}(u_{\varepsilon_j}-\phi)\pm\varepsilon_j \longrightarrow 0 \quad \text{on the set}\,\,\,  \left\{u<\phi - \frac{1}{k}\right\} \quad \text{as} \,\,\,j\rightarrow\infty.
$$
Therefore,
$$
   |D u|^\gamma F(x, D^2 u)  = 0 \quad \text{in} \quad \left\{u < \phi- \frac{1}{k}\right\} \quad (\text{in the viscosity sense})
$$
for each $k \in \mathbb{N}$ via stability of viscosity solutions (see, \cite[Corollary 2.7 and Remark 2.8]{BeDe14}). Moreover, since $\displaystyle \{u < \phi\} = \bigcup_{k=1}^{\infty} \left\{u<\phi - \frac{1}{k}\right\}$, we get in the viscosity sense
$$
  |D u|^\gamma F(x, D^2 u)  = 0 \quad \text{in} \quad \{u < \phi\}.
$$

Next let us consider the open set $\mathcal{O} \defeq \{u < \phi\}$ and suppose for the sake of contradiction that $\mathcal{O} \neq \emptyset$. From hypothesis under $\phi$ we have (in the viscosity sense)
{\small{
$$
   |D \phi|^\gamma F(x, D^2 \phi) \geq 0 \geq |D u|^\gamma F(x, D^2 u) \,\,\, \text{in} \,\,\, \mathcal{O} \,\,\, \stackrel[]{\text{Lemma \ref{cutting}}}{\Longrightarrow} \,\,\, F(x, D^2 \phi) \geq 0 \geq F(x, D^2 u) \quad \text{in} \quad \mathcal{O}
$$}}
and $u= \phi$ on $\partial \mathcal{O}$. From  Comparison Principle we get that $u \geq \phi$ in $\mathcal{O}$, which clearly is a contradiction. Thus, we conclude that $\mathcal{O} = \emptyset$ and $u \geq \phi$ in $B_1$.

Finally, we will show that
$$
|D u|^\gamma F(x, D^2 u) = f(x) \quad \text{in} \quad \{u > \phi\}
$$
in the viscosity sense. For that purpose, notice that for each $k \in \mathbb{N}$ we have that
$$
  f(x)\Phi_{\varepsilon_j}(u_{\varepsilon_j}-\phi) \pm \varepsilon_j\longrightarrow f(x) \quad \text{a.e. on the set}\,\,\,  \left\{u>\phi + \frac{1}{k}\right\}.
$$
Therefore, we conclude (in the viscosity sense) that
$$
\displaystyle   |D u|^\gamma F(x, D^2 u) = f(x) \quad \text{in} \quad \{u > \phi\} = \bigcup_{k=1}^{\infty} \left\{u>\phi + \frac{1}{k}\right\} \quad \text{as} \quad j \to \infty
$$
via stability of notion of viscosity solutions (see, \cite[Corollary 2.7 and Remark 2.8]{BeDe14}) and the result is proved.

In conclusion, $u$ fulfils the following \textit{a priori} estimate
$$
\begin{array}{rcl}
  \displaystyle \|u\|_{C^{1,\vartheta}(\overline{B_1})} & \leq &  \displaystyle \liminf_{j \to \infty} \|u_{\varepsilon_j}\|_{C^{1,\vartheta}(\overline{B_1})} \\
   & \leq & C(\vartheta).\left(\|u\|_{L^{\infty}(B_1)} + \|g\|_{C^{1, \zeta}(\partial B_1)}+\|f\|_{L^{\infty}(B_1)}^{\frac{1}{\gamma+1}}\right).
\end{array}
$$
\end{proof}

As an immediate consequence we obtain uniqueness of existing solutions.
\begin{cor}[{\bf Uniqueness, \cite[Corollary 2.4]{DSV}}]
The viscosity solution found in the Theorem \ref{Thm1} is unique.
\end{cor}

Finally, we would like to emphasize that in the homogeneous scenario, i.e. $f \equiv 0$, we do have uniqueness of existing solutions (cf. \cite[Theorem 3.3]{BLOP}) thanks to the fact that
$$
   |D w|^\gamma F(x, D^2 w) = 0 \,\,\, \text{in} \,\,\, B_1 \,\,\, \stackrel[]{\text{Lemma \ref{cutting}}}{\Longrightarrow} \,\,\, F(x, D^2 w)  = 0 \quad \text{in} \quad B_1.
$$

\subsection*{Acknowledgements.} This work was partially supported by Consejo Nacional de Investigaciones Cient\'{i}ficas y T\'{e}cnicas (CONICET-Argentina), Coordena\c{c}\~{a}o de Aperfei\c{c}oamento de Pessoal de N\'{i}vel Superior (PNPD/UnB-Brazil) Grant No. 88887.357992/2019-00 and CNPq-Brazil under Grant No. 310303/2019-2. J.V. da Silva thanks FCEyN/CEMIM from Universidad Nacional de Mar del Plata for its warm hospitality and for fostering a pleasant scientific atmosphere during his visit where part of this manuscript was written.

The authors would like to thank the anonymous Referee whose insightful comments and suggestions benefit a lot the final outcome this manuscript.

\end{document}